\documentclass[12pt]{amsproc}

\usepackage{amsfonts}
\usepackage{amsmath}
\usepackage{amssymb}

\usepackage{mathrsfs}
\usepackage{amscd}
\usepackage[all]{xy}
\usepackage[pdftex,final]{graphicx}
\usepackage{hyperref}
\usepackage[hyperpageref]{backref}

\usepackage{exscale,txfonts}

\newtheorem{lem}{Lemma}[section]
\newtheorem{thm}[lem]{Theorem}
\newtheorem{prop}[lem]{Proposition}
\newtheorem{cor}[lem]{Corollary}
\theoremstyle{definition}

\newtheorem{exa}[lem]{Example}
\newtheorem{rem}[lem]{Remark}

\newcommand{\Q}{\Bbb{Q}}
\newcommand{\F}[1]{\Bbb{F}_{#1}}

\newcommand{\Z}{\Bbb{Z}}

\newcommand{\R}{\Bbb{R}}

\newcommand{\g}{{\gamma}}

\newcommand{\spl}[2]{\mathrm{SL}_{#1}(#2)}

\newcommand{\gl}[2]{\mathrm{GL}_{#1}(#2)}

\newcommand{\lsl}[2]{\mathrm{sl}_{#1}(#2)}

\newcommand{\pcong}[2]{\Gamma(#1,#2)}
\newcommand{\tcong}[2]{\Gamma_{0}(#1,#2)}
\newcommand{\ucong}[2]{\Gamma_{1}(#1,#2)}
\newcommand{\tpcong}[2]{\tilde{\Gamma}(#1,#2)}
\newcommand{\card}[1]{\left| #1 \right|}

\renewcommand{\forall}{\mbox{ for all }}

\newcommand{\psyl}[2]{{#1}_{(#2)}}

\renewcommand{\ker}[1]{\mathrm{Ker}(#1)}

\newcommand{\coker}[1]{\mathrm{Coker}(#1)}

\newcommand{\sgr}[1]{\mathrm{R}_{#1}}% group ring of square classes of field
\newcommand{\an}[1]{\left\langle{#1}\right\rangle}
\newcommand{\pf}[1]{\left\langle\!\left\langle{#1}\right\rangle\!\right\rangle}
\newcommand{\gw}[1]{\mathrm{GW}(#1)}%Grothendieck-Witt ring
\newcommand{\aug}[1]{\mathcal{I}_{#1}}
\newcommand{\matr}[4]{\left[\begin{array}{cc}
#1&#2\\
#3&#4\\
\end{array}
\right]}

\newcommand{\bor}{B}

\newcommand{\sgn}[1]{\mathrm{sgn}(#1)}

\newcommand{\diag}[3]{\mathrm{diag}(#1,#2,#3)}
\newcommand{\ntr}[1]{\mathcal{O}_{#1}}

\newcommand{\cl}[1]{\mathrm{Cl}(#1)}

\newcommand{\mwk}[2]{K^{\mathrm{\small MW}}_{#1}({#2})}

\newcommand{\milk}[2]{K^{\mathrm{\small M}}_{#1}({#2})}

\newcommand{\mil}[2]{ \{ #1,#2 \} }

\newcommand{\st}[2]{\mathrm{St}(#1,#2)}
\newcommand{\sti}[1]{\mathrm{St}(#1)}

\newcommand{\ho}[3]{\mathrm{H}_{#1}(#2,#3 )}
\newcommand{\hoz}[2]{\ho{#1}{#2}{\Z}}
\newcommand{\coh}[3]{\mathrm{H}^{#1}(#2,#3 )}

\title{The second homology  of $\mathrm{SL}_2$ of 
 $S$-integers}
\author{Kevin Hutchinson}
\address{School of Mathematical Sciences,
 University College Dublin}
\email{kevin.hutchinson@ucd.ie}
\date{\today}

\keywords{
$K$-theory, Group Homology
}
\subjclass{19G99, 20G10}

\begin{document}
\bibliographystyle{plain}
\maketitle

\begin{abstract}
We calculate the structure of the finitely generated groups 
$\hoz{2}{\spl{2}{\Z[1/m]}}$  when 
$m$ is a multiple of $6$. 
Furthermore, we show how to construct  homology classes, represented by cycles in the 
bar resolution, which generate these groups and have prescribed orders.
 When $n\geq 2$ and 
$m$ is the product of the first 
$n$ primes, we combine our results with those of Jun Morita to 
show that the projection $\st{2}{\Z[1/m]}\to \spl{2}{\Z[1/m]}$ is the universal central 
extension.
Our methods have 
wider applicability: The main result on the structure of the second homology of certain 
rings is valid for rings of $S$-integers with sufficiently many units. 
For a wide class of rings $A$, we construct 
explicit homology classes in $\hoz{2}{\spl{2}{A}}$, functorially 
dependent on a pair of units, which correspond to symbols in $K_2(2,A)$. 

\end{abstract}

\section{Introduction}\label{sec:intro}
We calculate the structure of the finitely generated groups 
$\hoz{2}{\spl{2}{\Z[1/m]}}$  -- the Schur multiplier of $\spl{2}{\Z[1/m]}$ -- when 
$m$ is a multiple of $6$ (Theorem \ref{thm:z} below). 
Furthermore, we show how to construct explicit homology classes, in the 
bar resolution, which generate these groups and have prescribed orders 
(sections \ref{sec:classes} and 
\ref{sec:h2}). Our methods have 
wider applicability, however: The main result on the structure of the second homology of certain 
rings is valid for rings of $S$-integers with sufficiently many units. The homology classes 
which we construct make sense over any ring in which $6$ is a unit.

For a ring $A$ satisfying some finiteness  conditions the homology groups $\hoz{2}{\spl{n}{A}}$ 
are naturally isomorphic to the $K$-theory group $K_2(A)$ when $n$ is sufficiently large. 
However, $n=2$ is rarely sufficiently large, even when $A$ is a field. 

We review some background results (see Milnor \cite{mil:intro} for details).  For a 
commutative ring $A$, the \emph{unstable} $K_2$-groups
of the ring $A$, $K_2(n,A)$,  are defined to be the kernel of a surjective homomorphism
$\st{n}{A}\to E_n(A)$ where $\st{n}{A}$ is the rank $n-1$ Steinberg group of $A$ 
and where $E_n(A)$ is the 
subgroup of $\spl{n}{A}$ generated by elementary matrices.
There are compatible homomorphisms $\st{n}{A}\to \st{n+1}{A}$, $E_n(A)\to E_{n+1}(A)$, and 
taking direct limits as $n\to\infty$, we obtain a surjective map $\sti{A}\to E(A)$ whose kernel 
is $K_2(A):=\lim K_2(n,A)$. In fact, $K_2(A)$ is central in $\sti{A}$ and the extension 
\[
1\to K_2(A)\to \sti{A}\to E(A)\to 1
\]
is the universal central extension of $E(A)$ and hence $\hoz{2}{E(A)}\cong K_2(A)$.
 Furthermore, for a commutative ring $A$, 
$E(A)=\mathrm{SL}(A)=\lim \spl{n}{A}$.  

When $A$ satisfies some reasonable finiteness conditions 
these statements remain true when $K_2(A)$, $\sti{A}$ and $E(A)$ are replaced with 
$K_2(n,A)$, $\st{n}{A}$ and $E_n(A)$ for all sufficiently large $n$. In particular, when 
$F$ is a field with at least $10$ elements $\hoz{2}{\spl{2}{F}}\cong  K_2(2,F)$. 

When $F$ is  a global field  and when $S$ is a nonempty set of primes of $F$ containing 
the infinite primes, we let $\ntr{S}$ denote the corresponding ring of $S$-integers. (For 
example if $F=\Q$ and $1<m\in \Z$, we have $\Z[1/m]=\ntr{S}$ where $S$ consists of the 
primes dividing $m$ and the infinite prime.) Now the groups $\hoz{2}{\spl{2}{\ntr{S}}}$ and 
$K_2(2,\ntr{S})$ are finitely-generated abelian groups which satisfy 
\[
\lim_S \hoz{2}{\spl{2}{\ntr{S}}}=\hoz{2}{\spl{2}{F}}\mbox{ and }
\lim_S K_2(2,\ntr{S})=K_2(2,F).
\] 
It is natural to guess that we might have $\hoz{2}{\spl{2}{\ntr{S}}}\cong K_2(2,\ntr{S})$ 
when $S$ is sufficiently large in some appropriate sense. The example of $\ntr{S}=\Z$, 
when $\hoz{2}{\spl{2}{\Z}}=0$ while $K_2(2,\Z)\cong\Z$ shows that some condition on $S$ will 
be required. 

In the current paper, rather than comparing $\hoz{2}{\spl{2}{\ntr{S}}}$ to $K_2(2,\ntr{S})$ 
directly, 
we introduce a convenient proxy for $K_2(2,\ntr{S})$ which we denote 
$\tilde{K}_2(2,\ntr{S})$ (see section \ref{sec:main} below 
for definitions). There are natural maps 
\[ 
\hoz{2}{\spl{2}{\ntr{S}}}\to \tilde{K}_2(2,\ntr{S}),\quad
 K_2(2,\ntr{S})\to \tilde{K}_2(2,\ntr{S})
\]
 and the structure of the group $\tilde{K}_2(2,\ntr{S})$ is easy 
to describe (see Lemma \ref{lem:tilde}): 
\[
\tilde{K}_2(2,\ntr{S})\cong K_2(\ntr{S})_+\oplus \Z^r
\]
where $K_2(\ntr{S})_+$ is the subgroup of totally positive elements of $K_2(\ntr{S})$ 
and $r$ is the number of real embeddings of $F$. 

Our main theorem (\ref{thm:s}) states that when $S$ is sufficiently large (see the 
statement for more details) that the 
natural map $ \hoz{2}{\spl{2}{\ntr{S}}}\to \tilde{K}_2(2,\ntr{S})$ is an isomorphism. In the 
case $F=\Q$, the condition that $S$ be sufficiently large reduces to the requirement that 
$2,3\in S$. In particular, when $6|m$, we obtain  isomorphisms
\[
\hoz{2}{\spl{2}{\Z[1/m]}}\cong \tilde{K}_2(2,\Z[1/m])\cong \Z\oplus\left( \oplus_{p|m}\F{p}^\times
\right).
\] 

Jun Morita (\cite{morita:zs}) proved isomorphisms of the form 
\[
K_2(2,\Z[1/m])\cong \tilde{K}_2(2,\Z[1/m])
\]
for certain integers $m$ (eg. if $m$ is the product of the first $n$ prime numbers). 
Combining Morita's results with those above we deduce that 
\[
\hoz{2}{\spl{2}{\Z[1/m]}}\cong K_2(2,\Z[1/m])
\] 
for such $m$, and that, consequently, the extension
\[
1\to K_2(2,\Z[1/m])\to \st{2}{\Z[1/m]}\to \spl{2}{\Z[1/m]}\to 1
\]
is a universal central extension.

The main tool we use to prove Theorem \ref{thm:s} is the expression of 
$\spl{2}{\ntr{S\cup \{\mathfrak{p}\}}}$ as an amalgamated product 
\[
\spl{2}{\ntr{S}}\ast_{\tcong{\ntr{S}}{\mathfrak{p}}} H(\mathfrak{p})
\] 
associated to the action of $\spl{2}{\ntr{S\cup \{\mathfrak{p}\}}}$ on the Serre tree corresponding 
to the discrete valuation of the prime ideal $\mathfrak{p}$. This decomposition gives a 
Mayer-Vietoris sequence in homology. Analysis of the terms and the maps in low dimension 
yields, for $S$ sufficiently large, an exact sequence 
\[
\xymatrix{
\hoz{2}{\spl{2}{\ntr{S}}}\ar[r]
&\hoz{2}{\spl{2}{\ntr{S\cup \{\mathfrak{p}\}}}}\ar^-{\delta}[r]
& \hoz{1}{k(\mathfrak{p})} \to 0.
}
\]
where the map $\delta$ is essentially the tame symbol of $K$-theory (see Theorem \ref{thm:mv}). 
This analysis requires, in particular, the deep and beautiful theorem of Vaserstein and Liehl 
(\cite{vaserstein:sl2} and \cite{liehl}) and the solution of the congruence subgroup problem 
for $\mathrm{SL}_2$ (Serre, \cite{serre:sl2}). 

In the later part of the paper, we tackle an old question in $K_2$-theory; namely, how 
to write down natural homology classes in $\hoz{2}{\spl{2}{A}}$, 
depending functorially on a pair of units $u,v\in A^\times$, which correspond, 
under the map $\hoz{2}{\spl{2}{A}}\to K_2(2,A)$ when it exists, to the symbols 
$c(u,v)\in K_2(2,A)$.  The answer to the corresponding question for 
$\hoz{2}{\spl{3}{A}}$ and $K_2(3,A)$ is well-known, namely the homology class (in the bar 
resolution)
\[
\left([\diag{u}{u^{-1}}{1}|\diag{v}{1}{v^{-1}}]-[\diag{v}{1}{v^{-1}}|\diag{u}{u^{-1}}{1}]
\right)\otimes 1
\]
corresponds to the symbol $\mil{u}{v}\in K_2(3,A)$, at least up to sign. 
There is no such simple expression in 
the case of
$K_2(2,A)$. The symbols $c(u,v)$ are easily and naturally described 
in terms of the generators of the Steinberg group, but the corresponding natural 
homology classes, even
in the case of a field, have no known simple construction. Since $K_2(2,\Z)$ is infinite cyclic 
with generator $c(-1,-1)$ while $\hoz{2}{\spl{2}{\Z}}=0$ it follows 
that there can be no simple universal expression defined over the ring $\Z$. The homology 
classes, $C(u,v)$, that we construct in section \ref{sec:classes} below are not very elegant 
(though it seems unlikely that they can be greatly improved on). To begin with, 
the construction of the representing cycles 
requires the presence of a unit $\lambda$ such that $\lambda^2-1$ 
is also a unit, although the resulting homology classes can be shown quite generally to be independent 
of the choice of $\lambda$. Furthermore, the representing cycles consist usually of 
$32$ terms and hence are far from simple. 

However, the cycles we construct are explicit and functorial for  homomorphisms of rings. We prove 
(Theorem \ref{thm:symb}) that they map to the symbols $c(u,v)\in K_2(2,A)$ when $A$ is a field.
We can thus use them  to write down provably non-trivial homology classes in 
$\hoz{2}{\spl{2}{A}}$ for more general rings $A$. 
In particular, in section \ref{sec:h2}, we use them to write down explicit elements 
of the groups $\hoz{2}{\spl{2}{\ntr{S}}}$ with given order and to construct generators 
of the groups $\hoz{2}{\spl{2}{\Z[1/m]}}$ when $m$ is divisible by $6$.

\section{Preliminaries and notation}
\subsection{Notation}
For a Dedekind Domain $A$ with field of fractions $F$, $\cl{A}$ denotes the ideal 
classgroup of $A$. If $\mathfrak{p}$ 
is a nonzero prime ideal of $A$, 
$v_{\mathfrak{p}}:F^\times\to\Z$ denotes the corresponding discrete value. 
For a global field $F$ and a nonempty set of primes $S$ of $F$ we let $\ntr{S}$ 
denote the ring of $S$-integers:
\[
\ntr{S}:=\{ a\in F^\times\ |\ v_{\mathfrak{p}}(a)\geq 0\mbox{ for all }\mathfrak{p}\not\in S\}.
\]

For a finite abelian group $M$, $\psyl{M}{p}$ denotes the Sylow $p$-subgroup of $M$. 

For a commutative ring $A$, we let $\sgr{A}:= \Z[A^\times/(A^\times)^2]$ be the group ring of the group of 
square classes of units.
 For $a\in A^\times$, the square class of $a$ will be denoted $\an{a}\in \sgr{A}$. Furthermore, the 
element $\an{a}-1$ in the augmentation ideal, $\aug{A}\subset \sgr{A}$, will be denoted $\pf{a}$. 

\subsection{Elementary matrices}
We will have occasion to refer to  the following facts:

For a commutative ring $A$, and any $x\in A$ we define the elementary matrices 
\[
E_{12}(x):=\matr{1}{x}{0}{1}, E_{21}(x):=\matr{1}{0}{x}{1}\in \spl{2}{A}.
\]
Let $E_2(A)$ be the subgroup of $\spl{2}{A}$ generated 
by  $E_{12}(x)$, $E_{21}(y)$, $x,y\in A$.

The following theorem of Vaserstein and Liehl will be essential below. Its proof relies on the 
resolution of the congruence subgroup problem for $\mathrm{SL}_2$ (see Serre \cite{serre:sl2}).

\begin{thm}[Vaserstein \cite{vaserstein:sl2}, Liehl \cite{liehl}]\label{thm:vl}
Let $K$ be a global field and let $S$ be a set of places of $K$ of cardinality at least 
$2$ and containing 
all archimedean places. Let 
\[
\mathcal{O}_S:= \{ x\in K\ |\ v(x)\geq 0 \mbox{ for all }v\not\in S\} 
\]
be the ring of $S$-integers of $K$. Let $I_1$ and $I_2$ be nonzero ideals of $\mathcal{O}_S$. 
Let
\[
\tilde{\Gamma}(I_1,I_2):= \left\{ \matr{a}{b}{c}{d}\in \spl{2}{\mathcal{O}_S}\ | \ 
b\in I_1,c\in I_2,  a-1,b-1\in I_1I_2\right\}
\]
Then $\tilde{\Gamma}(I_1,I_2)$ is generated by the  elementary matrices
\[
E_{12}(x), x\in I_1\mbox{ and } E_{21}(y), y\in I_2.
\] 

\end{thm}

\begin{prop} \label{prop:e2} 
Let $A$ be a commutative ring.
\begin{enumerate}
\item $E_2(A)=\spl{2}{A}$ if $A$ is a field or a Euclidean domain or if $A=\mathcal{O}_S$ is the 
ring of $S$-integers in a global field and $|S|\geq 2$. 
\item $E_2(A)$ is perfect if there exists $\lambda_1,\ldots,\lambda_n\in A^\times$ and $b_1,\ldots, 
b_n\in A$ such that\\
 $\sum_{i=1}^nb_i(\lambda_i^2-1)=1$ in $A$. 

In particular, $E_2(A)$ is perfect if there exists $\lambda\in A^\times$ such that $\lambda^2-1\in 
A^\times$ also.
\end{enumerate}
\end{prop}
\begin{proof}
\begin{enumerate}
\item This is standard linear algebra in the case of a Euclidean Domain or a field, 
and the theorem of Vaserstein-Liehl in the case of $S$-integers.
\item For $\lambda\in A^\times$, let 
\[
D(\lambda):=\matr{\lambda}{0}{0}{\lambda^{-1}}\in\spl{2}{A}.
\]

Note that $D(\lambda)\in E_2(A)$ since 
\[
D(\lambda)=w(\lambda)w(-1)\mbox{ where } w(\lambda):=
\matr{0}{\lambda}{-\lambda^{-1}}{0}=E_{12}(\lambda)E_{21}(-\lambda^{-1})E_{12}(\lambda).
\]
Then 
\[
D(\lambda)E_{12}(x)D(\lambda)^{-1}= E_{12}(\lambda^2x)
\]
and hence, for any $b\in A$ we have 
\[
[D(\lambda),E_{12}(bx)]=D(\lambda)E_{12}(bx)D(\lambda)^{-1}E_{12}(-bx)=E_{12}((\lambda^2-1)bx).
\]

Thus 
\[
E_{12}(x)=E_{12}(\sum_i(\lambda_i^2-1)b_ix)=\prod_iE_{12}((\lambda_i^2-1)b_ix)=
\prod_i[D(\lambda_i),E_{12}(b_ix)].
\]
\end{enumerate}
\end{proof}
\begin{rem}On the other hand, the groups $E_2(\F{2})=\spl{2}{\F{2}}$ and $E_2(\F{3})=\spl{2}{\F{3}}$ 
are not perfect. It follows that if the ring $A$ admits a homomorphism to $\F{2}$ or $\F{3}$ then 
$E_2(A)$ is not perfect. In particular, the group $E_2(\Z)$ is not perfect.  
\end{rem}
\begin{rem}
In \cite{swan:special}, R. Swan showed that $E_2(A)\not=\spl{2}{A}$ for  
$A=\Z[\sqrt{-5}]$.

Indeed, when $A$ is the ring of integers in a quadratic imaginary number field then 
$E_2(A)\not=\spl{2}{A}$ except in the five cases that $A$ is a Euclidean Domain (see 
\cite{vaserstein:sl2}). 
\end{rem}

\subsection{Homology of Groups}
For any group $G$, $F_\bullet(G)$ will denote the (right) bar 
resolution of $\Z$ over $\Z[G]$: i.e. 
for $n\geq 1$, $F_n(G)$ is the free right $\Z[G]$-module with generators $[g_n|\cdots |g_1]$, $\g_i\in G$, 
and $F_0(G)=\Z[G]$ (regarded as a right $\Z[G]$-module). The boundary homomorphism 
$d_n:F_n(G)\to F_{n-1}(G)$ is given by 
\[
d_n([g_n|\cdots |g_1])=[g_n|\cdots |g_2]g_1+\sum_{i=1}^{n-1}(-1)^{n-i}[g_{n-1}|\cdots |g_{i+1}g_{i}|\cdots |g_1]
+(-1)^n[g_{n-1}|\cdots|g_1]
\]
for $n\geq 2$ and $d_1([g]):= g-1$.

We let $\bar{F}_\bullet(G)$ denote the complex $\{\bar{F}_n(G)\}_{n\geq 0}$ where 
\[
\bar{F}_n(G):= F_n(G)\otimes_{\Z[G]}\Z.
\]  
Thus $\hoz{n}{G}\cong H_n(\bar{F}_\bullet(G))$.

We will require the following standard ``centre kills'' argument from group homology:

\begin{lem}\label{lem:ck}
Let $G$ be a group and let $M$ be a $\Z[G]$-module. Suppose that $g\in Z(G)$ has the property that 
$g-1$ acts as an automorphism on $M$. Then $\ho{i}{G}{M}=0$ for all $i\geq 0$. 
\end{lem}

\section{The functor $K_2(2,A)$}\label{sec:k2}
In this section, we review some of the theory of the functor $K_2(2,A)$ 
%-- or \emph{rank one $K_2$} -- 
for commutative rings $A$. 
%It is often closely related to $\hoz{2}{\spl{2}{A}}$.

\subsection{Definitions}
Let $A$ be a commutative ring.  

We let $A^\times$ act by automorphisms 
on $\spl{2}{A}$ as follows: Let 
\[
M(a):=\matr{a}{0}{0}{1}\in \gl{2}{A}.
\] 
and define
\[
a\ast X:= X^{M(a)}=M(a)^{-1}XM(a)
\]
for $a\in A^\times$, $X\in \spl{2}{A}$.

In particular, we have 
\[
a\ast E_{12}(x)= E_{12}(a^{-1}x)\mbox{ and } a\ast E_{21}(x)=E_{21}(ax)
\]
for all $a\in A^\times$, $x\in A$.

The rank one Steinberg group $\st{2}{A}$ is defined by generators and relations as follows:
The generators are the terms 
\[
x_{12}(t)\mbox{ and } x_{21}(t), \quad t\in A
\]
and the defining relations are 
\begin{enumerate}
\item 
\[
x_{ij}(s)x_{ij}(t)=x_{ij}(s+t)
\]
for $i\not=j\in \{ 1,2\}$ and all $s,t\in A$, and 
\item For $u\in A^\times$, let 
\[
w_{ij}(u):= x_{ij}(u)x_{ji}(-u^{-1})x_{ij}(u)
\]
for $i\not=j\in \{ 1,2\}$. Then 
\[
w_{ij}(u)x_{ij}(t)w_{ij}(-u)=x_{ji}(-u^{-2}t)
\]
for all $u\in A^\times$, $t\in A$.
\end{enumerate}

There is a natural surjective homomorphism $\phi:\st{2}{A}\to E_2(A)$ defined by 
$\phi(x_{ij}(t))=E_{ij}(t)$ for all $t$.  It is easily verified that the formulae 
\[
a\ast x_{12}(t)=x_{12}(a^{-1}t)\mbox{ and }a\ast x_{21}(t)=x_{21}(at)
\]
define an action of $A^\times$ on $\st{2}{A}$ by automorphisms. Clearly the homomorphism 
$\phi$ is equivariant with respect to this action.

By definition $K_2(2,A)$ is the kernel of $\phi$. It inherits an action of $A^\times$. 

For $u\in A^\times$ and for $i\not=j\in \{ 1,2\}$, we let 
\[
h_{ij}(u):= w_{ij}(u)w_{ij}(-1).
\]
Note that 
\[
\phi(w_{12}(u))=\matr{0}{u}{-u^{-1}}{0} \mbox{ and } \phi(h_{12}(u))=\matr{u}{0}{0}{u^{-1}}.
\]

Note that, from the definitions and defining relation (1), 
for any $a\in A$ and for any unit $u$ we have 
\[
x_{ij}(a)^{-1}=x_{ij}(-a)\mbox{ and } w_{ij}(u)^{-1}=w_{ij}(-u). 
\]

The defining relation  (2) above thus immediately gives the following 
conjugation formula. 
 
\begin{lem}\label{lem:conjw}
Let $A$ be a commutative ring. Let $a\in A$ and $u\in A^\times$. For $i\not=j\in 
\{ 1,2\}$ 
\[
x_{ij}(a)^{w_{ij}(-u)}=x_{ji}(-u^{-2}a).
\]
\end{lem}

Since the right-hand-side is unchanged by $u\rightarrow -u$, we deduce:
\begin{cor}\label{cor:conjw}
Let $A$ be a commutative ring. Let $a\in A$ and $u\in A^\times$. For $i\not=j\in 
\{ 1,2\}$ 
\[
x_{ij}(a)^{w_{ij}(u)^{-1}}=x_{ji}(-u^{-2}a)=x_{ij}(a)^{w_{ij}(u)}.
\]
and 
\[
x_{ji}(a)^{w_{ij}(u)}=x_{ij}(-u^2a).
\]
\end{cor}

Form the definition of $h_{ij}(u)$, we then obtain:
\begin{cor}\label{cor:conjh}
Let $A$ be a commutative ring. Let $a\in A$ and $u\in A^\times$. For $i\not=j\in 
\{ 1,2\}$ 
\[
x_{ij}(a)^{h_{ij}(u)}=x_{ij}(u^{-2}a)\mbox{ and }
x_{ij}(a)^{h_{ij}(u)^{-1}}=x_{ij}(u^2a).
\]
\end{cor}

\subsection{Symbols}
In particular, for $u,v\in A^\times$ the \emph{symbols} 
\[
c(u,v):= h_{12}(u)h_{12}(v)h_{12}(uv)^{-1}
\]
lie in $K_2(2,A)$. 

The elements $c(u,v)$ are central in $\st{2}{A}$. We let $C(2,A)$ denote the 
subgroup of $K_2(2,A)$ generated by these symbols. 

Note that for $a,u\in A^\times$ we have 
\[
a\ast w_{12}(u)=w_{12}(a^{-1}u)\mbox{ and } a\ast w_{21}(u)=w_{21}(au)
\]
and hence 
\[
a\ast h_{12}(u)=h_{12}(a^{-1}u)h_{12}(a^{-1})^{-1}\mbox{ and }
a\ast h_{21}(u)=h_{21}(au)h_{21}(a)^{-1}. 
\]

It follows easily that 
\[
a\ast c(u,v) =c(u,a^{-1})^{-1}c(u,a^{-1}v).
\]

Thus the abelian group $C(2,A)$ is a module over the group ring 
$\Z[A^\times]$ with this action. 

\begin{lem}\label{lem:symb} Let $A$ be a commutative ring. Then 
\[
a^2\ast c(u,v)=c(u,v)
\]  
for all $a,u,v\in A^\times$. 

In particular, $C(2,A)$ is naturally an $\sgr{A}$-module.
\end{lem}
\begin{proof}
We have $h_{ij}(u)=h_{ji}(u)^{-1}$ in $\st{2}{A}$. Thus 
\[
c(u,v)= h_{12}(u)h_{12}(v)h_{12}(uv)^{-1}=h_{21}(u)^{-1}h_{21}(v)^{-1}h_{21}(uv).
\]

Thus 
\begin{eqnarray*}
a*c(u,v)&=& h_{21}(a)h_{21}(au)^{-1}h_{21}(a)h_{21}(av)^{-1}h_{21}(auv)h_{21}(a)^{-1}\\
&=&  h_{21}(a)h_{21}(au)^{-1}h_{21}(a)\cdot\left( h_{21}(u)h_{21}(u)^{-1}\right)\cdot 
h_{21}(av)^{-1}h_{21}(auv)h_{21}(a)^{-1}\\
&=& h_{21}(a)c(u,a)^{-1}c(u,av)h_{21}(a)^{-1}\\
&=& c(u,a)^{-1}c(u,av)\\
&=& a^{-1}\ast c(u,v).\\
\end{eqnarray*}
\end{proof}
The symbols $c(u,v)$ satisfy the following properties (see \cite{mat:pres}, or also 
\cite{steinberg:chev})
\begin{prop}\label{prop:mat}
Let $A$ be a commutative ring. Then
\begin{enumerate}
\item $c(u,v)=1$ if $u=1$ or $v=1$.
\item $c(u,v)=c(v^{-1},u)$ for all $u,v\in A^\times$.
\item $c(u,vw)c(v,w)=c(uv,w)c(u,v)$ for all $u,v,w\in A^\times$.
\item $c(u,v)=c(u,-uv)$ for all $u,v\in A^\times$
\item $c(u,v)=c(u,(1-u)v)$ whenever $u,1-u,v\in A^\times$. 
\end{enumerate}
\end{prop} 

\begin{rem}
Combining the result of Lemma \ref{lem:symb} with Proposition \ref{prop:mat} (3), we see 
that the square class $\an{a}\in \sgr{A}$ acts on $C(2,A)$ via
\[
\an{a}c(u,v)= c(u,a)^{-1}c(u,av)= c(au, v)c(a,v)^{-1}.
\]   

Futhermore Proposition \ref{prop:mat} (4) is equivalent to 
\[
\an{v}c(u,-u)=1\mbox{ for all }u,v\in A^\times
\]
and Proposition \ref{prop:mat} (5) is equivalent to
\[
\an{v}c(u,1-u)=1 \mbox{ for all }u,v\in A^\times.
\] 
\end{rem}

We will use the following property of symbols (\cite{mat:pres}):
\begin{lem}\label{lem:square}
 If $u,v,w$ are units in $A$, then 
\[
c(u,v^2w)=c(u,v^2)c(u,w)
\]
and
\[
c(u,v^2)=c(u,v)c(v,u)^{-1}=c(u^2,v).
\]
\end{lem}

Furthermore, we have the following theorem of Matsumoto and Moore (\cite{mat:pres},\cite{moore:pres}):

\begin{thm}\label{thm:mm}
Let $F$ be an infinite field. Then 
\begin{enumerate}
%\item $C(2,F)=K_2(2,F)$ and hence $K_2(2,F)$ is central in $\st{2}{F}$.
\item The sequence 
\[
1\to K_2(2,F)\to \st{2}{F}\to \spl{2}{F}\to 1
\]
is the universal central extension of the perfect group $\spl{2}{F}$. 

In particular, 
$K_2(2,F)\cong\hoz{2}{\spl{2}{F}}$ naturally. 

\item $K_2(2,F)$ has the following presentation: It is generated by the symbols $c(u,v)$, 
$u,v\in F^\times$, subject to the five relations of Proposition \ref{prop:mat}.
\end{enumerate}
\end{thm}

\subsection{The stabilization homomorphism $K_2(2,F)\to K_2(F)$}
For a field  $F$, 
the Theorem of Matsumoto also gives a presentation of $K_2(n,F)$ for all $n\geq 3$. 
In particular, it follows that $K_2(F)=\milk{2}{F}$, the second Milnor $K$-group of the field 
$F$. The stabilization map $K_2(2,F)\to K_2(F)$ is surjective and sends the symbols 
$c(u,v)$ to the symbols  $\{ u,v\}$ of algebraic $K$-theory.  

Let $\gw{F}$ be the Grothendieck-Witt ring of isometry classes of nondegerate quadratic forms 
over $F$. It is generated by the classes $\an{a}$ of $1$-dimensional forms and the 
map $\sgr{F}\to\gw{F}$ sending $\an{a}\to \an{a}$ is a surjection of rings. The fundamental 
ideal $I(F)$ of $\gw{F}$ is the ideal generated by the elements $\pf{a}:=\an{a}-1$.  

There is a natural surjective homomorphism of $\sgr{F}$-modules 
\[
K_2(2,F)\to I^2(F),\quad c(u,v)\mapsto \pf{u}\pf{v}.
\]

Furthermore, by a theorem of Milnor (\cite{milnor:quad}) there is also a surjective map 
$\milk{2}{F}\to I^2(F)/I^3(F)$ sending the symbol $\{ u,v\}$ to the class of 
$\pf{u}\pf{v}$.  The kernel of this map is precisely $2\milk{2}{F}$. 

By a result essentially due to 
Suslin (\cite{sus:tors}, but see also \cite{mazz:sus})  
for an infinite field $F$, we also have the following description of $K_2(2,F)$:
\begin{thm}\label{thm:susmazz} Let $F$ be an infinite field.
The maps $K_2(2,F)\to K_2(F)$, $K_2(2,F)\to I^2(F)$ induce an isomorphism of $\sgr{F}$-modules
\[
K_2(2,F)\to \milk{2}{F}\times_{I^2(F)/I^3(F)}I^2(F), 
\quad c(u,v)\mapsto [u,v]:= (\mil{u}{v},\pf{u}\pf{v}).
\]
\end{thm}

\begin{cor}\label{cor:mwk}
 Let $F$ be an infinite field.
There is a natural short exact sequence of $GW(F)$-modules
\[
0\to I^3(F)\to K_2(2,F)\to \milk{2}{F}\to 0.
\]
\end{cor}

\subsection{Milnor-Witt $K$-theory}\label{sec:mwk}
The homology of the special linear group of a field is related to the Milnor-Witt $K$-theory 
of the field (see, for example, \cite{hutchinson:tao3}). 

Milnor-Witt $K$-theory of a field $F$ is a $\Z$-graded algebra $\mwk{\bullet}{F}$ generated 
by symbols $[ u]$, $u\in F^\times$ in degree $1$ and a symbol $\eta$ in degree $-1$, satisfying 
certain relations (see \cite{morel:trieste} for details). It arises naturally as a ring of 
operations in stable $\mathbb{A}^1$-homotopy theory.

A deep theorem of Morel asserts:
\begin{thm}\label{thm:morel}[\cite{morel:puiss}]
There is a natural isomorphism of graded rings 
\[
\mwk{\bullet}{F}\cong \milk{\bullet}{F}\times_{I^{\bullet}(F)/I^{\bullet +1}(F)}I^\bullet(F).
\]

(Here, when $n<0$, $\mwk{n}{F}:=0$ and $I^n(F):=W(F)$, the Witt ring of the field.)
\end{thm}

The theorem of Suslin on the structure of $K_2(2,F)$ quoted above, implies 
\begin{prop}
There is a natural isomorphism $K_2(2,F)\cong \mwk{2}{F}$, sending $c(u,v)$ to 
$[u][v]$. 
\end{prop}

\section{The map from $\hoz{2}{\spl{2}{F}}$ to $K_2(2,F)$}  \label{sec:map}
Let $A$ be a commutative ring for which $E_2(A)=\spl{2}{A}$ is a perfect group. Suppose further that 
the group extension
\[
\xymatrix{
1\ar[r]
& K_2(2,A)\ar[r]
& \st{2}{A}\ar^-{\phi}[r]
& \spl{2}{A}\ar[r]
& 1
}
\]
is a central extension.

Let $s:\spl{2}{A}\to \st{2}{A}$ be a section of $\phi$. Then there is a corresponding 
$2$-cocycle $f_s:\spl{2}{A}\times \spl{2}{A}\to K_2(2,A)$ defined by 
\[
f_s(X,Y):= s(X)s(Y)s(XY)^{-1}.
\] 
This yields a cohomology class $f\in \coh{2}{\spl{2}{A}}{K_2(2,A)}$ which is independent of 
the choice of section $s$. 

However, since $\hoz{1}{\spl{2}{A}}=0$, the universal coefficient theorem tells us that 
there is a natural isomorphism 
\[
\coh{2}{\spl{2}{A}}{K_2(2,A)}\cong \mathrm{Hom}(\hoz{2}{\spl{2}{A}},K_2(2,A))
\]
described as follows: Let $z\in \coh{2}{\spl{2}{A}}{K_2(2,A)}$ be represented by the $2$-cocycle 
$h$. Then $h$ induces a homomorphism 
\[
\bar{F}_2\to K_2(2,F),\quad  \sum_in_i[X_i|Y_i]\mapsto \prod_ih(X_i,Y_i)^{n_i}
\]
which vanishes on boundaries, and thus in turn induces a homomorphism
\[
\bar{h}: \hoz{2}{\spl{2}{A}}\to K_2(2,F).
\]

In particular, the cocycle $f_s$ above induces the homomorphism 
\[
\hoz{2}{\spl{2}{A}}\to K_2(2,F), \quad  \sum_in_i[X_i|Y_i]\mapsto \prod_if_s(X_i,Y_i)^{n_i}
\] 

This homomorphism is an isomorphism precisely when the central extension is universal. In particular,
it is an isomorphism when $A$ is an infinite field, by the theorem of Matsumoto-Moore.

We now specialise to the case of a  field $F$. 

For our calculations, we will use the following section $s:\spl{2}{F}\to \st{2}{F}$:
\[
s\left(\matr{a}{b}{c}{d}\right):=
\left\{
\begin{array}{ll}
x_{12}(ab)h_{12}(a), & \mbox{ if } c=0,\\
x_{12}(ac^{-1})w_{12}(-c^{-1})x_{12}(dc^{-1}), & \mbox{ if }c\not=0.\\
\end{array}
\right.
\]

Note that, in particular, we have 
\[
s(E_{ij}(a))=x_{ij}(a)\mbox{ and } s(D(u))=h_{12}(u)
\]
when $i\not= j\in \{ 1,2\}$, $a\in A$ and $u\in F^\times$. 

Furthermore, functoriality of the constructions above 
guarantee that the induced homomorphism
\[
\bar{f}:\hoz{2}{\spl{2}{F}}\to K_2(2,F)
\]
is a map of $\Z[F^\times]$-modules. 
Recall that this homomorphism is induced by the homomorphism
\[
\bar{F}_2(\spl{2}{F})\to K_2(2,F), [X|Y]\mapsto f_s(X,Y)=s(X)s(Y)s(XY)^{-1}.
\] 

\begin{lem}\label{lem:f}
 Let $F$ be a field. Let $u,v\in F^\times$ and $a,b\in F$. 
 Let
\[
X=\matr{u}{a}{0}{u^{-1}}, \quad Y= \matr{v}{b}{0}{v^{-1}}
\]
Then 
$
f_s(X,Y)=c(u,v).
$
\end{lem}

\begin{proof} We have,
\[
s(X)=x_{12}(au)h_{12}(u),\quad s(Y)=x_{12}(bv)h_{12}(v)\mbox{ and } s(XY)= x_{12}(bu^2v+au)h_{12}(uv).
\]

Thus 
\begin{eqnarray*}
f(X,Y)&=& x_{12}(au)h_{12}(u)x_{12}(bv)h_{12}(v)h_{12}(uv)^{-1}x_{12}(-bu^2v-au)\\
&=& x_{12}(au)x_{12}(bv)^{h_{12}(u)^{-1}}h_{12}(u)h_{12}(v)h_{12}(uv)^{-1}x_{12}(-bu^2v-au)\\
&=& x_{12}(au)x_{12}(bu^2v)c(u,v)x_{12}(-bu^2v)x_{12}(-au)\mbox{\quad by Corollary 
\ref{cor:conjh}}\\
&=& c(u,v)\mbox{\quad since $c(u,v)$ is central.} \\
\end{eqnarray*}

\end{proof}

\begin{cor}\label{cor:square}
Let $F$ be a field. Let $a,b\in F^\times$. Then 
\[
\left([D(a)|D(b)]-[D(b)|D(a)]\right)\otimes 1\in F_2(\spl{2}{F})\otimes \Z 
\]
is a cycle and the corresponding homology class maps to $c(a^2,b)$ under the natural 
isomorphism $\hoz{2}{\spl{2}{k}}\cong K_2(2,F)$ induced by $f_s$.
\end{cor}
\begin{proof}
The first statement in immediate since $D(a)D(b)=D(ab)=D(b)D(a)$. 

The image of this cycle is 
\[
f_s(D(a),D(b))\cdot f_s(D(b),D(a))^{-1}=c(a,b)c(b,a)^{-1}=c(a^2,b)
\]
by Lemma \ref{lem:square}. 
\end{proof}

\section{The Mayer-Vietoris  sequence}

Throughout this section $A$ will denote a Dedekind Domain with field of fractions $K$.
% and $\mathfrak{p}$ will be a nonzero prime ideal in $A$.
\subsection{The groups $H(I)$} We collect together some basic and well-known facts about 
certain subgroups of $\spl{2}{K}$ (see for example \cite[p. 520]{serre:sl2}).

Let $I$ be a fractional ideal of $A$.

We consider the lattice $\Lambda=\Lambda_I:=A\oplus I \subset K\oplus K=K^2$.

Let $H(I)$ denote the the subgroup 
\[
\{ M\in \spl{2}{K}\ | M\cdot \Lambda= \Lambda\}=
\left\{ \matr{a}{b}{c}{d}\in \spl{2}{K}\ |\ a,d\in A, c\in I,d\in I^{-1}
\right\}=
\tilde{\Gamma}(I,I^{-1}).
\]

Note that, in particular, $H(A)=\spl{2}{A}$.

We also note that if $J$ is any nonzero fractional ideal of $A$, then
\[
H(I)= \{ M\in \spl{2}{K}\ | M\cdot (J\Lambda)= J\Lambda\}
\] 
where 
\[
J\Lambda=J\cdot(A\oplus I)=J\oplus IJ. 
\]

%In particular, since $I^{-1}\cdot \Lambda_I=I^{-1}\oplus I$, $H(I^2)$ is the 
\begin{lem}\label{lem:hi} Let $I$ be a fractional ideal of the $A$.
\begin{enumerate}
\item Suppose that $I'=aI$ where $0\not=a\in K$. Then $H(I')=H(I)^{M(a)}$ where 
\[
M(a)=\matr{a}{0}{0}{1}\in\gl{2}{K}.
\] 
\item Suppose $I$ is an  integral ideal. Let 
\[
A'= \{ r\in K\ | v_{\mathfrak{q}}(r)\geq 0\mbox{ for all }\mathfrak{q}\not|I\}.
\]
Then there exists $M\in\spl{2}{A'}$ such that 
$H(I^2)=\spl{2}{A}^M$. In particular, $H(I^2)\cong \spl{2}{A}$. 
\end{enumerate}
\end{lem}
\begin{proof}\ 

\begin{enumerate}
\item This follows from the observation that multiplication by $M(a)$ induces an isomorphism 
of lattices $A\oplus I'\cong a\cdot(A\oplus I)$, and hence conjugation by $M(a)$ induces an 
isomorphism of the stabilizers.

\item We first observe that, since $I^{-1}\cdot \Lambda_I=I^{-1}\oplus I$, 
$H(I^2)$ is the stabilizer of $I^{-1}\oplus I$. 

There exists an integral ideal $J$ of $A$ satisfying: $I+J=A$ and $IJ=xA$ for some nonzero $x\in A$. 
So $J=xI^{-1}$.  
Thus multiplication by $M(x)$ induces an isomorphism $I^{-1}\oplus I\cong J\oplus I$. 

Choose $a\in I, b\in J$ with $a+b=1$. Consider the short exact sequence of $A$-modules 
\[
\xymatrix{
0\ar[r]
&xA\ar^-{f}[r]
&J\oplus I\ar^{g}[r]
&A\ar[r]
&0\\ 
}
\] 
where $g(y)=(y,-y)$ and $f(y,z)=y+z$. There is a splitting $A\to J\oplus I$ given by 
$y\mapsto (by,ay)$. This gives an isomorphism of $A$-modules 
\[
J\oplus I\cong xA\oplus A,\quad (y,z)\mapsto (ay-bz,y+z); 
\]
i.e.multiplication by 
\[
N:= \matr{a}{-b}{1}{1}\in\spl{2}{A}
\]
induces an isomorphism of lattices $J\oplus I\cong xA\oplus A$.

Now, multiplication by $M(x)^{-1}$ induces an isomorphism $xA\oplus A\cong A\oplus A$. 

Putting all of this together, multiplication by 
\[
M:=M(x)^{-1}NM(x)=\matr{a}{-b/x}{x}{1}\in\spl{2}{K}
\]
induces an isomorphism of lattices $I^{-1}\oplus I\cong A\oplus A$, and thus conjugation by $M$ 
induces an isomorphism of stabilizers as required.

Finally, we note that since $xA=IJ$ and $bA=JK$ for some integral ideal $K$, $(b/x)A=KI^{-1}$ and 
hence $b/x\in A'$. Thus 
$M\in\spl{2}{A'}$ as claimed. 
\end{enumerate}
\end{proof}
\begin{cor}
Let $I$ be a fractional ideal of $A$. Suppose that the class of $I$ in 
%the ideal classgroup of $A$ 
$\cl{A}$ is a square. Then $H(I)\cong \spl{2}{A}$.
\end{cor}

\begin{rem}
In particular, the ideal $\tilde{\mathfrak{p}}:=\mathfrak{p}A_{\mathfrak{p}}$ 
in $A_{\mathfrak{p}}$ is a principal ideal 
with generator $\pi$, say. It follows from Lemma \ref{lem:hi} that 
\[
H(\tilde{\mathfrak{p}})= \spl{2}{A_{\mathfrak{p}}}^{M(\pi)}.
\]
\end{rem}

Let $\mathfrak{p}$ be a nonzero prime ideal of $A$. Let $n\geq 1$ and 
let $\pi\in A$ satisfy 
$v_{\mathfrak{p}}(\pi)=1$. We let $\gamma_{\pi,n}:H(\mathfrak{p})\to \spl{2}{A/\mathfrak{p}^n}$ be 
the composite 
\[
\xymatrix{
H(\mathfrak{p})\ar[r]
&H(\tilde{\mathfrak{p}})\ar^-{\cong}_-{\mathrm{conj}_{M(1/\pi)}}[r]
&\spl{2}{A_{\mathfrak{p}}}\ar[r]
&\spl{2}{A_{\mathfrak{p}}/\tilde{\mathfrak{p}}^n}\ar[r]^-{\cong}
&\spl{2}{A/\mathfrak{p}^n}
}
\]
\begin{lem} The map $\gamma_{\pi,n}$ is surjective for all $n$ and the kernel of this map is 
independent of the choice of $\pi$. 
\end{lem}

\begin{proof}
By definition, we have 
\[
\gamma_{\pi,n}\left(\matr{a}{b}{c}{d}\right)=\matr{\bar{a}}{\bar{\pi b}}{\bar{c/\pi}}{\bar{d}}
\]
where 
\[
\bar{x}:= x+\tilde{\mathfrak{p}}^n\in A_{\mathfrak{p}}/\tilde{\mathfrak{p}}^n\cong A/\mathfrak{p}^n.
\]

Since $\spl{2}{A/\mathfrak{p}^n}$ is generated by elementary matrices, we need only show how to 
lift these. We begin by observing that $\pi A=\mathfrak{p}J$ where $J$ is an ideal not contained 
in $\mathfrak{p}$. It follows that $A=\mathfrak{p}^n+J$ for any $n\geq 1$; i.e. the map 
$J\to A/\mathfrak{p}^n$ is surjective. 

Thus, given any $x\in A$ there exists $x'\in J$ with $\bar{x'}=\bar{x}$. Since $x'\in 
J$ it follows that $x'/\pi\in J\cdot (\mathfrak{p}J)^{-1}=\mathfrak{p}^{-1}$. Hence 
$E_{12}(x'/\pi)\in H(\mathfrak{p})$ and 
\[
\gamma_{\pi,n}(E_{12}(x'/\pi))= E_{12}(\bar{x'})= E_{12}(\bar{x}).
\]
Of course, we also have $E_{21}(\pi x)\in H(\mathfrak{p})$ and 
$\gamma_{\pi,n}(E_{21}(\pi x))=E_{21}(\bar{x})$. This proves the surjectivity statement.

For the second part, suppose that $\pi'\in A$ also satisfies $v_{\mathfrak{p}}(\pi')=1$. 
Then $\pi'=\pi\cdot u$ for some $u\in A_{\mathfrak{p}}^\times$. From the definition, we have 
\[
\gamma_{\pi',n}= f\circ \gamma_{\pi,n}
\] 
where $f$ is conjugation by $M(\bar{u}^{-1})$ on $\spl{2}{A/\mathfrak{p}^n}$. It follows at 
once that $\ker{\gamma_{\pi',n}}=\ker{\gamma_{\pi,n}}$ as claimed.
\end{proof}

We let $\tpcong{A}{\mathfrak{p}^n}$ denote the kernel of the $\gamma_{\pi,n}$ (for any 
choice of $\pi$). Thus, for all $n\geq 1$, there is a short exact sequence 
\[
1\to \tpcong{A}{\mathfrak{p}^n}\to H(\mathfrak{p})\to \spl{2}{A/\mathfrak{p}^n}\to 1.
\]

Note that
\[
\tpcong{A}{\mathfrak{p}^n}=
\left\{ \matr{a}{b}{c}{d}\in H(\mathfrak{p})\ |\ a-1,d-1\in \mathfrak{p}^n, c\in \mathfrak{p}^{n+1},
b\in \mathfrak{p}^{n-1}\right\}.
\]
In particular, for all $n\geq 1$ we have 
\[
\pcong{A}{\mathfrak{p}^{n+1}}\subset \tpcong{A}{\mathfrak{p}^n}\subset 
\tcong{A}{\mathfrak{p}^n}\subset \spl{2}{A}.
\]

For a field $F$, we will use the notation
\[
\bor(F):= \left\{ \matr{a}{b}{0}{a^{-1}}\in \spl{2}{F}\right\}\mbox{ and }
\bor'(F):= \left\{ \matr{a}{0}{c}{a^{-1}}\in \spl{2}{F}\right\}.
\]
Of course, these two subgroups of $\spl{2}{F}$ are naturally isomorphic. 

We will need the following result below.

\begin{lem}\label{lem:bor'}
 There is a natural short exact sequence 
\[
1\to \tpcong{A}{\mathfrak{p}}\to \tcong{A}{\mathfrak{p}}\to \bor'(k(\mathfrak{p}))\to 1.
\]
\end{lem}
\begin{proof}
This is immediate from the fact that the image of $\tcong{A}{\mathfrak{p}}$ in 
$\spl{2}{A/\mathfrak{p}}=\spl{2}{k(\mathfrak{p})}$ under the map $\gamma_{\pi,1}$ is precisely 
$\bor'(k(\mathfrak{p}))$. 
\end{proof}
\subsection{The Mayer-Vietoris sequence}
Let $\mathfrak{p}$ be a  nonzero prime ideal of $A$ and 
let $v=v_{\mathfrak{p}}$ be the associated discrete valuation.  We let $k(\mathfrak{p})$
 or $k(v)$ denote  the
residue field  $A/\mathfrak{p}$.   We will further suppose that 
the class of $\mathfrak{p}$ has finite order in $\cl{A}$. Thus $\mathfrak{p}^n=xA$ for some 
$n\geq 1$ and $x\in A$. (This condition is automatically satisfied when $K$ is  a global field.) 

Let
\[
\pcong{A}{\mathfrak{p}}:=\ker{\spl{2}{A}\to \spl{2}{k(\pi)}}=
\left\{ \matr{a}{b}{c}{d}\in\spl{2}{A}\ :\ 1-a,1-d,b,c\in\mathfrak{p}\right\}
\]  
and let
\[
\tcong{A}{\mathfrak{p}}:= \left\{ \matr{a}{b}{c}{d}\in\spl{2}{A}\ :\ c\in\mathfrak{p}\right\}.
\]

We let $\tilde{\mathfrak{p}}$ denote the extension of $\mathfrak{p}$ to the localization 
$A_{\mathfrak{p}}$, which is thus a discrete valuation ring with unique (principal) nonzero prime ideal 
 $\tilde{\mathfrak{p}}$.

The action of $\spl{2}{K}$ on the Serre tree associated to the valuation $v$ 
(\cite[Chapter II]{serre:trees}) 
yields a 
decomposition 
\begin{eqnarray}\label{amal1}
\spl{2}{K}=\spl{2}{A_{\mathfrak{p}}}\star_{\tcong{A_{\mathfrak{p}}}{\tilde{\mathfrak{p}}}}
H(\tilde{\mathfrak{p}})
\end{eqnarray}
of $\spl{2}{K}$ as the sum of $\spl{2}{A_{\mathfrak{p}}}$ and $H(\tilde{\mathfrak{p}})$
amalgamated  along 
their intersection\\
 $\spl{2}{A_{\mathfrak{p}}}\cap H(\tilde{\mathfrak{p}})
= 
\tcong{A_{\mathfrak{p}}}{\tilde{\mathfrak{p}}}$.

Let 
\[
A':=\{ a\in K\ |\ v_{\mathfrak{q}}(a)\geq 0 \mbox{ for all prime ideals }\mathfrak{q}\not=\mathfrak{p}
\}.  
\]
Note that since $\mathfrak{p}^n=xA$ by assumption, $A'=A[1/x]$.

Since $A[1/x]$ is dense in $K$ in the $\mathfrak{p}$-adic topology, and since 
\[
\spl{2}{A[1/x]}\cap\spl{2}{A_{\mathfrak{p}}}=\spl{2}{A},\quad 
\spl{2}{A[1/x]}\cap H(\tilde{\mathfrak{p}})= H(\mathfrak{p}) 
\]
there is also an induced decomposition
\[
\spl{2}{A[1/x]}=\spl{2}{A}\star_{\tcong{A}{\mathfrak{p}}}H(\mathfrak{p}).
\]

For convenience, in the remainder of this section we will set 
\[
G:=\spl{2}{A[1/x]},\  
G_1:= \spl{2}{A},\ G_2:= H(\mathfrak{p})\mbox{ and } \Gamma_0:= \tcong{A}{\mathfrak{p}}.
\]
%(Of course, it follows that conjugation by $M(\pi)$ induces an isomorphism $G_1\cong G_2$.) 
Thus $G=G_1\star_{\Gamma_0}G_2$ and this 
decomposition gives rise a short exact sequence of $\Z[G]$-modules:
\[
\xymatrix{
0\ar[r]
&\Z[G/\Gamma_0]\ar^-{\alpha}[r]
&\Z[G/G_1]\oplus \Z[G/G_2]\ar^-{\beta}[r]
&\Z\ar[r]
&0.}
\]
where $\alpha$ is the map 
\[
\alpha: \Z[G/\Gamma_0]\to \Z[G/G_1]\oplus \Z[G/G_2],\  g\Gamma_0\mapsto (gG_1,gG_2)
\]
and $\beta$ is the unique $\Z[G]$-homomorphism
\[
\beta: \Z[G/G_1]\oplus \Z[G/G_2]\to \Z,\ (G_1,0)\mapsto -1, (0,G_2)\mapsto 1.
\]

This short exact sequence of $\Z[G]$-modules gives rise to a long exact sequence in 
homology. Combining this with the isomorphisms of Shapiro's lemma,
$\ho{r}{G}{\Z[G/H]}\cong\hoz{r}{H}$, gives us the \emph{Mayer-Vietoris exact sequence}
of the amalgamated product:
\[
\xymatrix{
\cdots \ar^-{\delta}[r]
&\hoz{r}{\Gamma_0}\ar^-{\alpha}[r]
&\hoz{r}{G_1}\oplus\hoz{r}{G_2}\ar^-{\beta}[r]
&\hoz{r}{G}\ar^-{\delta}[r]
&\cdots
}
\]  

The maps $\alpha$ and $\beta$ in this sequence can be described as follows: Let 
$\iota_1:\Gamma_0\to G_1$ and $\iota_2:\Gamma_0\to G_2$ be the natural inclusions. Then 
\[
\alpha(z)=(\iota_1(z),\iota_2(z))\mbox{ for all } z\in \hoz{r}{\Gamma_0}.
\] 
Likewise, let $j_1:G_1\to G$ and $j_2:G_2\to G$ be the natural inclusions. Then 
\[
\beta(z_1,z_2)=j_2(z_2)-j_1(z_1)\mbox{ for all } z_1\in\hoz{r}{G_1}, \ z_1\in\hoz{r}{G_1}.
\]

The amalgamated product decomposition (\ref{amal1}) -- i.e. taking the case $A=A_{\mathfrak{p}}$ -- 
also gives rise to a Mayer-Vietoris sequence
\[
\xymatrix{
\cdots \ar^-{\delta}[r]
&\hoz{r}{\tcong{A_{\mathfrak{p}}}{\tilde{\mathfrak{p}}}}\ar^-{\alpha}[r]
&\hoz{r}{\spl{2}{A_{\mathfrak{p}}}}\oplus\hoz{r}{H(\tilde{\mathfrak{p}})}\ar^-{\beta}[r]
&\hoz{r}{\spl{2}{K}}\ar^-{\delta}[r]
&\cdots
}
\]

\subsection{The connecting homomorphism}
As above, let $\mathfrak{p}$ be a prime ideal of the Dedekind Domain $A$ and let 
\[
\delta:\hoz{2}{\spl{2}{K}}\to \hoz{1}{\tcong{A_{\mathfrak{p}}}{\tilde{\mathfrak{p}}}}
\]
be the connecting homomorphism in the Mayer-Vietoris sequence associated to the decomposition
\[
\spl{2}{K}=\spl{2}{A_{\mathfrak{p}}}\star_{\tcong{A_{\mathfrak{p}}}{\tilde{\mathfrak{p}}}}
H(\tilde{\mathfrak{p}})
=\spl{2}{A_{\mathfrak{p}}}\star_{\tcong{A_{\mathfrak{p}}}{\tilde{\mathfrak{p}}}}\spl{2}{A_{\mathfrak{p}}}^{M(\pi)}.
\]

\begin{prop}\label{prop:delta}
Let $\rho:\tcong{A_{\mathfrak{p}}}{\tilde{\mathfrak{p}}}\to 
k(\tilde{\mathfrak{p}})^\times$ be the (surjective) map
\[
\matr{a}{b}{c}{d}\mapsto a\pmod{\tilde{\mathfrak{p}}}.
\]
Then the composite homomorphism, $\Delta$ say, 
\[
\xymatrix{
K_2(2,K)\cong \hoz{2}{\spl{2}{K}}\ar^-{\delta}[r]
&\hoz{1}{\tcong{A_{\mathfrak{p}}}{\tilde{\mathfrak{p}}}}\ar^-{\rho}[r]
&\hoz{1}{k(\tilde{\mathfrak{p}})^\times}\cong k(\tilde{\mathfrak{p}})^\times 
}
\]
is the map 
\[
c(a,b)\mapsto (-1)^{v(a)v(b)}\frac{b^{v(a)}}{a^{v(b)}}\pmod{\tilde{\mathfrak{p}}}.
\]
\end{prop}
\begin{rem} In fact, the isomorphisms in the statement of Proposition \ref{prop:delta} are canonical 
only up to sign. We have made our choices so that the sign is $+1$; but the choice of sign does not 
materially affect our main results. 
\end{rem}

Before proving Proposition \ref{prop:delta}, we require

\begin{lem}\label{lem:k2gen}
Let $K$ be a field with discrete valuation $v$. Then $K_2(2,K)$ is generated by the set 
$C_v:=\{ c(x,u)\ |\ v(u)=0, v(x)=1\}$. 
\end{lem}
\begin{proof} Let $D$ be the subgroup of $K_2(2,K)$ generated by $C_v$. Let $a,b\in K^\times$. 
We must prove that $c(a,b)\in D$. 

Since 
\[
c(a,b)=c(b^{-1},a)=c(a^{-1},b^{-1})=c(b,a^{-1})
\]
we can assume that $v(a),v(b)\geq 0$. 

We will prove the result by induction on $n=v(a)+v(b)\geq 0$.

If $n=0$, then $v(a)=v(b)=0$ and choosing $\pi\in K^\times$ with $v(\pi)=1$ we have 
\[
c(a,b)= c(\pi a,b)^{-1}c(\pi a,b)c(\pi,a)\in D. 
\]

On the other hand, suppose that $v(a),v(b)>0$. If $0<v(b)\leq v(a)$ then $a=bc$ with 
$0\leq v(c)<v(a)$ and hence 
\[
c(a,b)=c(bc,b)=c(-c,b)\in D
\]
by the inductive hypothesis. An analogous argument applies to the case $0<v(a)<v(b)$. 

Since $c(a,b)=c(b^{-1},a)$, we can reduce to the case  where $v(b)=0$ and $v(a)\geq 2$. 
Then let $a=a'\pi$ where $v(\pi)=1$ and $1\leq v(a')<v(a)$. We have 
\[
c(a,b)=c(a'\pi,b)= c(a',\pi b)c(\pi,b)c(a',\pi)^{-1}
\]
which lies in $D$ by the induction hypothesis (using the argument for the case 
$v(a),v(b)>0$ for the first term). 
\end{proof}

\begin{proof}[Proof of Proposition \ref{prop:delta}]
By Lemma \ref{lem:k2gen}, we must prove that 
\[
\Delta(c(x,u))= u\pmod{\tilde{\mathfrak{p}}}
\]
whenever $v(u)=0$, $v(x)=1$.

We note that it is enough to prove that $\Delta(c(x,u^2))=u^2\pmod{\tilde{\mathfrak{p}}}$ whenever  
$v(u)=0$, $v(x)=1$. For if $u\in K$ is not a square, choose an extension $v'$ of $v$ to 
$K':=K(\root {} \of {u})$. Then there is a natural map of Mayer Vietoris exact sequences 
inducing a commutative square
\[
\xymatrix{
\hoz{2}{\spl{2}{K}}\ar^-{\delta}[r]\ar[d]
&k(v)^\times\ar@{(->}^-{i}[d]\\
\hoz{2}{\spl{2}{K'}}\ar^-{\delta'}[r]
&k(v')^\times
}
\]
so that $i(\Delta(c(x,u))=\Delta'(c(x,u))=\bar{u}\in k(v')^\times$ since $u$ is a square in $K'$, 
and thus $\Delta(c(x,u))=\bar{u}\in k(v)^\times$.

Now, by Corollary \ref{cor:square},
 the symbol $c(x,u^2)\in K(2,K)$ corresponds to the homology class represented 
by the cycle
\[
Z:=\left([D(x)|D(u)]-[D(u)|D(x)]\right)\otimes 1 \in {F}_2(G)\otimes_{\Z[G]} \Z
\]
where $G=\spl{2}{K}$. 

Recall that the Mayer-Vietoris sequence is the long exact homology sequence derived from
the short exact sequence of complexes
\[
0\to F_\bullet(G)\otimes_{\Z[G]}\Z[G/\Gamma_0]\to
F_\bullet(G)\otimes_{\Z[G]}\left(\Z[G/G_1]\oplus\Z[G/G_2]\right)\to 
F_\bullet(G)\otimes_{\Z[G]}\Z\to 0.
\]

Now the cycle $Z$ lifts to 
\[
\left([D(x)|D(u)]-[D(u)|D(x)]\right)\otimes (1\cdot G_1,0)\in F_2(G)\otimes 
\left(\Z[G/G_1]\oplus\Z[G/G_2]\right).
\]

Under the boundary map $d_2$, this is sent to 
\[
[D(u)]\otimes (D(x)\cdot G_1-1\cdot G_1,0)\in
F_1(G)\otimes 
\left(\Z[G/G_1]\oplus\Z[G/G_2]\right)
\]
since $D(u)\in \Gamma_0\subset G_1$. 

Now let 
\[
w:=\matr{0}{1}{-1}{0}\in G_1.
\]
Then
\[
w\cdot D(x)=w^{M(x)}\in G_2.
\]
Thus $(D(x)\cdot G_1-1\cdot G_1,0)$ is the image of 
\[
w^{-1}\cdot (w^{M(x)}\Gamma_0-\Gamma_0)=D(x)\Gamma_0-w^{-1}\Gamma_0
\]
under the map 
$\alpha:\Z[G/\Gamma_0]\to \Z[G/G_1]\oplus\Z[G/G_2]$. 

Thus the homology class $\delta(Z)\in \hoz{1}{\Gamma_0}$ is represented by the cycle
\[
[D(u)]\otimes (D(x)\Gamma_0-w^{-1}\Gamma_0)
=\left([D(u)]D(x)-[D(u)]w^{-1}\right)\otimes  \Gamma_0 \in
F_1(G)\otimes_{\Z[G]}\Z[G/\Gamma_0]. 
\] 

This, in turn, is the image of 
\[
\left([D(u)]D(x)-[D(u)]w^{-1}\right)\otimes 1 \in
F_1(G)\otimes_{\Z[\Gamma_0]}\Z
\]
under the natural isomorphism 
\[
F_\bullet(G)\otimes_{\Z[\Gamma_0]}\Z\cong F_\bullet(G)\otimes_{\Z[G]}\Z[G/\Gamma_0].
\]
%In order to invert the isomorphism of Shapiro's Lemma, we switch to the homogeneous resolution. 

For a group $H$ we let $C_\bullet(H)$ denote the right homogeneous resolution of $H$. The isomorphism 
$F_\bullet(H)\to C_\bullet(H)$ of complexes of right $\Z[H]$-modules is given by 
\[
[h_n|\cdots |h_1]\mapsto (h_n\cdot h_{n-1}\cdots h_1,\ldots, h_1,1).
\]

Thus  the cycle $\left([D(u)]D(x)-[D(u)]w^{-1}\right)\otimes 1 \in
F_1(G)\otimes_{\Z[\Gamma_0]}\Z$ corresponds to the cycle
\[
\left((D(ux),D(u))-(D(u)w^{-1},w^{-1})\right)\otimes 1\in C_1(G)\otimes_{\Z[\Gamma_0]}\Z.
\]

To construct an augmentation-preserving map of $\Z[\Gamma_0]$-resolutions from 
$C_\bullet(G)$ to $C_\bullet(\Gamma_0)$, we choose any set-theoretic section $s:G/\Gamma_0\to G$ 
of the natural surjection $G\to G/\Gamma_0, g\mapsto g\Gamma_0$. For $g\in G$ we let 
$\bar{g}:=s(g\Gamma_0)^{-1}g\in \Gamma_0$.  Then the 
map 
\[
\tau:C_\bullet(G)\to C_\bullet(\Gamma_0), (g_n,\ldots,g_0)\mapsto (\bar{g}_n,\ldots,\bar{g}_0)
\]
is an augmentation preserving map of $\Z[\Gamma_0]$-complexes. 

We further specify that we the section $s$ satisfies 
\[
s(D(u)w^{-1}\Gamma_0)=w^{-1}\mbox{ and } s(D(x)\Gamma_0)=D(x)
\]
for all $u$ with $v(u)=0$. Then 
\[
\tau\left((D(ux),D(u))-(D(u)w^{-1},w^{-1})\right)=(D(u),1)-(D(u^{-1}),1)\in C_1(\Gamma_0)
\]
since $wD(u)w^{-1}=D(u^{-1})$ in $G$.

Finally, the homology class 
\[
\left((D(u),1)-(D(u^{-1}),1)\right)\otimes 1 \in C_1(\Gamma_0)\otimes_{\Z[\Gamma_0]}\Z
\] 
corresponds to the element 
\[
D(u)\cdot D(u^{-1})^{-1}=D(u^2)\in \Gamma_0/[\Gamma_0,\Gamma_0]
\]
under the isomorphism $\hoz{1}{\Gamma_0}\cong  \Gamma_0/[\Gamma_0,\Gamma_0]$, and hence maps to 
$u^2\pmod{\tilde{\mathfrak{p}}}\in k(\tilde{\mathfrak{p}})^\times$  under the map $\rho$. 
\end{proof}
\subsection{The abelianization of some congruence subgroups}
\begin{prop}\label{prop:tcong} Let $A$ be a ring of $S$-integers in a global field $K$. 
Suppose that 
$|S|\geq 2$ and that there exists $\lambda\in A^\times$ such that $\lambda^2-1\in A^\times$ also. 
Let 
$\mathfrak{p}$ be a nonzero prime ideal of $A$. 
  
Then the  map $\rho:\tcong{A}{\mathfrak{p}}\to k(\mathfrak{p})^\times$ induces an isomorphism
\[
\hoz{1}{\tcong{A}{\mathfrak{p}}}\cong k(\mathfrak{p})^\times.
\]
\end{prop}
\begin{proof}
The map $\rho$ induces a short exact sequence
\[
\xymatrix{
1\ar[r]
&\ucong{A}{\mathfrak{p}}\ar[r]
& \tcong{A}{\mathfrak{p}}\ar^-{\rho}[r]
&k(\mathfrak{p})^\times\ar[r]
&1
}
\]
where
\[
\ucong{A}{\mathfrak{p}}:=\left\{ \matr{a}{b}{c}{d}\in \spl{2}{A}\ |\ 
\c,a-1,d-1 \in \mathfrak{p}\right\}=\tilde{\Gamma}(A,\mathfrak{p})
\]
in the notation of Theorem \ref{thm:vl}. 

Since $k(\mathfrak{p})^\times$ is an abelian group, it follows that 
\[
[\tcong{A}{\mathfrak{p}},\tcong{A}{\mathfrak{p}}]\subset \ucong{A}{\mathfrak{p}}
\]

On the other hand, by Theorem \ref{thm:vl},  $\ucong{A}{\mathfrak{p}}=
\tilde{\Gamma}(A,\mathfrak{p})$ 
is generated by elementary matrices $E_{12}(x), x\in A$, $E_{21}(y), y\in \mathfrak{p}$. 
However,
\[
E_{12}(x)=[D(\lambda), E_{12}(x/(\lambda^2-1))], E_{21}(y)=[D(\lambda), E_{21}(y/(\lambda^2-1))\in 
[\tcong{A}{\mathfrak{p}},\tcong{A}{\mathfrak{p}}].
\]  

So $[\tcong{A}{\mathfrak{p}},\tcong{A}{\mathfrak{p}}]=\ucong{A}{\mathfrak{p}}$ as required. 
\end{proof}

For a field $k$  we let $\lsl{2}{k}$ denote the $3$-dimensional vector space of $2\times 2$ 
trace zero matrices. 

\begin{lem}\label{lem:congindex}
 Let $A$ be a Dedekind domain and let $\mathfrak{p}$ be a maximal ideal. Then 
for any $m\geq 1$ there are natural isomorphisms of groups 
\[
\frac{\tpcong{A}{\mathfrak{p}^n}}{\tpcong{A}{\mathfrak{p}^{n+1}}}\cong
\lsl{2}{k(\mathfrak{p})}\cong \frac{\pcong{A}{\mathfrak{p}^n}}{\pcong{A}{\mathfrak{p}^{n+1}}}
\]
\end{lem}
\begin{proof}
From the definitions of $\pcong{A}{\mathfrak{p}^n}$ and $\tpcong{A}{\mathfrak{p}^n}$
we have
\[
\frac{\tpcong{A}{\mathfrak{p}^n}}{\tpcong{A}{\mathfrak{p}^{n+1}}}\cong 
\ker{\spl{2}{A/\mathfrak{p}^{n+1}}\to \spl{2}{A/\mathfrak{p}^{n}}}\cong 
\frac{\pcong{A}{\mathfrak{p}^n}}{\pcong{A}{\mathfrak{p}^{n+1}}}.
\]

Let $\pi\in \mathfrak{p}\setminus\mathfrak{p}^2$. For any $n\geq 1$, the group 
$\mathfrak{p}^n/\mathfrak{p}^{n+1}$ is a $1$-dimensional 
$k(\mathfrak{p})$-vector spaces with basis $\{ \pi^n+\mathfrak{p}^{n+1}\}$. 

The required isomorphism 
\[
\lsl{2}{k(\mathfrak{p})}\to
\ker{\spl{2}{A/\mathfrak{p}^{n+1}}\to \spl{2}{A/\mathfrak{p}^{n}}}
\]
is then the  map 
\[
\matr{\bar{a}}{\bar{b}}{\bar{c}}{\bar{d}}\mapsto \matr{1+a\pi^n}{b\pi^n}{c\pi^n}{1+d\pi^n}
\]
where $a,b,c,d\in A$ map to $\bar{a},\bar{b},\bar{c},\bar{d}\in k(\mathfrak{p})$.  
\end{proof}
\begin{cor}\label{cor:congindex}
 Let $A$ be a Dedekind domain and let $\mathfrak{p}$ be a maximal ideal. Suppose that 
$k(\mathfrak{p})$ is a finite field  with $q$ elements. Then 
\[
[\tpcong{A}{\mathfrak{p}}:\tpcong{A}{\mathfrak{p}^n}]=q^{3(n-1)}=
[\pcong{A}{\mathfrak{p}}:\pcong{A}{\mathfrak{p}^n}]
\]
for all $n\geq 1$. 
\end{cor}

\begin{lem}\label{lem:crt}
 Suppose that $I$ and $J$ are comaximal ideals in $A$; i.e. $I+J=A$. Then 
the composite map $\pcong{A}{I}\to \spl{2}{A}\to \spl{2}{A/J}$ is surjective.
\end{lem}

\begin{proof}
By the Chinese Remainder Theorem the map $A/IJ\to (A/I)\times (A/J)$, $a\mapsto  (a+I,a+J)$ is 
 an isomorphism of rings. It follows that the map
\[
\spl{2}{A/IJ}\to \spl{2}{A/I}\times \spl{2}{A/J}, \quad X\mod{IJ}\mapsto (X\mod{I},X\mod{J})
\]
is an isomorphism of groups and hence that 
\[
\spl{2}{A}\to \spl{2}{A/I}\times \spl{2}{A/J}, \quad X\mapsto (X\mod{I},X\mod{J})
\]
is a surjective group homomorphism. This implies the statement of the Lemma.
\end{proof}

\begin{lem}\label{lem:index} Suppose that $k(\mathfrak{p})$ is a finite field 
with  $q$ elements. We have 
\[
[\spl{2}{A}:\pcong{A}{\mathfrak{p}}]=q(q^2-1)=
[\spl{2}{A}:\tpcong{A}{\mathfrak{p}}]
\]
\end{lem}

\begin{proof}
The first equality follows from the isomorphism 
\[
\frac{\spl{2}{A}}{\pcong{A}{\mathfrak{p}}}\cong \spl{2}{k(\mathfrak{p})}.
\]

For the second inequality, denote by $C$ the image of the map 
\[
\tpcong{A}{\mathfrak{p}}\to \spl{2}{A}\to \spl{2}{A/\mathfrak{p}^2}.
\]

Then $C$ fits into a short exact sequence 
\[
1\to W\to C\to T\to 1
\]
where 
\[
T=\left\{ \matr{a}{0}{0}{d}\in \spl{2}{A/\mathfrak{p}^2}\ | a-1,d-1\in \mathfrak{p}\right\} 
\cong k(\mathfrak{p})
\]
and 
\[
W= \left\{ \matr{1}{b}{0}{1}\in \spl{2}{A/\mathfrak{p}^2}\ \right \}\cong A/\mathfrak{p}^2.
\]

It follows that $\card{C}=[\tpcong{A}{\mathfrak{p}}:\pcong{A}{\mathfrak{p}^2}]=q^3$. Since 
$[\spl{2}{A}:\pcong{A}{\mathfrak{p}^2}]=\card{\spl{2}{A/\mathfrak{p}^2}}=q^4(q^2-1)$, the second equality 
follows.
\end{proof}
\begin{prop}\label{prop:pcong} Let $A$ be a ring of $S$-integers in a global field $K$ where 
$|S|\geq 2$. Let $\mathfrak{p}$ be a nonzero prime ideal and let $p>0$ 
be the characteristic of the 
residue field $k(\mathfrak{p})$. 

Suppose that $\mathfrak{p}^n=xA$ for some $n\geq 1$ and $x\in A$.  
Suppose further that there exist $\lambda\in A[1/x]^\times$ such 
that $\lambda^2-1\in A[1/x]^\times$ also.
 Then $\hoz{1}{\pcong{A}{\mathfrak{p}}}$ and $\hoz{1}{\tpcong{A}{\mathfrak{p}}}$ 
are  finite abelian $p$-groups. 
\end{prop}
\begin{proof}
Let $\Gamma$ denote either $\pcong{A}{\mathfrak{p}}$ or $\pcong{A}{\mathfrak{p}}$. 
By Proposition 2 of \cite{serre:sl2}
 the commutator subgroup $[\Gamma,\Gamma]$ contains a 
principal congruence subgroup $\pcong{A}{I}$  for some ideal $I$ of $A$. 
There exists a nonzero ideal $J$ of $A$ such that $I$ factors as 
$\mathfrak{p}^mJ$ where $J\not\subset\mathfrak{p}$ and $m\geq 1$. Since 
$\pcong{A}{\mathfrak{p}^{m+1}J}\subset \pcong{A}{\mathfrak{p}^mJ}$, we can suppose without loss that 
$m\geq 2$.

By definition, $\pcong{A}{I}$ is the kernel of the natural map $\spl{2}{A}\to \spl{2}{A/I}$. This 
map is surjective since $\spl{2}{A/I}=E_2(A/I)$. 

Since $J\not\subset \mathfrak{p}$, it follows that $A=\mathfrak{p}^m+ J$ 
and hence, by Lemma \ref{lem:crt}, the map $\pcong{A}{\mathfrak{p}^m}\to \spl{2}{A/J}$ 
is surjective. Since $\pcong{A}{\mathfrak{p}^m}\subset \Gamma$, it follows that the map
$\Gamma\to\spl{2}{A/J}$ is surjective. 

However, since $\spl{2}{A/J}=E_2(A/J)$ and since $x+J$ is a unit in $A/J$ the hypotheses 
of the proposition ensure that all elementary matrices are commutators and hence that 
$\spl{2}{A/J}$ is a perfect group. It then follows that the natural 
map $[\Gamma,\Gamma]\to \spl{2}{A/J}$ is surjective.

Thus $[\Gamma,\Gamma]/\pcong{A}{I}$ surjects onto $\spl{2}{A/J}$ and hence 
\[
\card{\spl{2}{A/J}}\mbox{ divides }\left[[\Gamma,\Gamma]:\pcong{A}{I}]\right].  
\]

It follows that 
\[
[\spl{2}{A}:[\Gamma,\Gamma]]|\card{\spl{2}{A/\mathfrak{p}^m}}=(q^2-1)q^{3m-2}.
\]

Since $[\spl{2}{A}:\Gamma]=q(q^2-1)$ by Lemma \ref{lem:index} it follows that 
$\card{\Gamma/[\Gamma,\Gamma]}$ divides $q^{3m-3}$, and so is a power of $p$ as claimed.
\end{proof}
\subsection{The second homology of congruence subgroups}
\begin{lem}\label{lem:borel}
Let $k$ be a finite field of characteristic $p$ and let $M$ be an $\spl{2}{k}$-module. 
Then, for all $i\geq 0$,  the natural map 
\[
\psyl{\ho{i}{\bor(k)}{M}}{p}\to \psyl{\ho{i}{\spl{2}{k}}{M}}{p}
\]
is an isomorphism.
\end{lem}
\begin{proof}
As in the proof of \cite[Cor 3.10.2]{hut:cplx13}.
\end{proof}

\begin{prop}\label{prop:h2}
Let $A$ be a ring of $S$-integers in a global field $K$ where 
$|S|\geq 2$. 
Let $\mathfrak{p}$ be a nonzero prime ideal and let $p>0$ be the characteristic of the 
residue field $k(\mathfrak{p})$. Suppose that $\mathfrak{p}^m=xA$ for some 
$m\geq 1$, $x\in A$. 
Suppose further that there exist $\lambda\in A[1/x]^\times$ such 
that $\lambda^2-1\in A[1/x]^\times$ also.

Then the natural maps 
\[
\iota_1:\hoz{2}{\tcong{A}{\mathfrak{p}}}\to\hoz{2}{\spl{2}{A}}
\]
and
\[
\iota_2:\hoz{2}{\tcong{A}{\mathfrak{p}}}\to\hoz{2}{H(\mathfrak{p})}
\]
are surjective.
\end{prop}
\begin{proof}
Let $k=k(\mathfrak{p})$.  
There is a commutative diagram  of group extensions
\[
\xymatrix{
1\ar[r]
&\pcong{A}{\mathfrak{p}}\ar[r]\ar^-{\mathrm{id}}[d]
&\tcong{A}{\mathfrak{p}}\ar[r]\ar^-{\iota_1}[d]
&\bor(k)\ar[r]\ar^-{\iota}[d]
&1\\
1\ar[r]
&\pcong{A}{\mathfrak{p}}\ar[r]
&\spl{2}{A}\ar[r]
&\spl{2}{k}\ar[r]
&1\\
}
\]
and  (using Lemma \ref{lem:bor'})
\[
\xymatrix{
1\ar[r]
&\tpcong{A}{\mathfrak{p}}\ar[r]\ar^-{\mathrm{id}}[d]
&\tcong{A}{\mathfrak{p}}\ar[r]\ar^-{\iota_2}[d]
&\bor'(k)\ar[r]\ar^-{\iota}[d]
&1\\
1\ar[r]
&\tpcong{A}{\mathfrak{p}}\ar[r]
&H(\mathfrak{p})\ar^-{\gamma_{\pi,1}}[r]
&\spl{2}{k}\ar[r]
&1.\\
}
\]

We give the argument for $\iota_1$. The analogous argument for $\iota_2$ is achieved by replacing 
$\bor(k)$ with $\bor'(k)$. 

The top group extension gives rise to a spectral sequence 
\[
E^2_{i,j}(\tcong{A}{\mathfrak{p}})=\ho{i}{\bor(k)}{\hoz{j}{\pcong{A}{\mathfrak{p}}}}\Rightarrow 
\hoz{i+j}{\tcong{A}{\mathfrak{p}}}
\]
and the lower one gives rise to the spectral sequence
\[
E^2_{i,j}(\spl{2}{A})=\ho{i}{\spl{2}{k}}{\hoz{j}{\pcong{A}{\mathfrak{p}}}}\Rightarrow 
\hoz{i+j}{\spl{2}{A}}.
\]
The map of extensions induces a natural map of spectral sequences compatible with 
the map $\iota_1$ on abutments.

For $H=\tcong{A}{\mathfrak{p}}$ or $\spl{2}{A}$, the image of the edge homomorphism 
$E^\infty_{0,j}(H)\to \hoz{j}{H}$ is equal to the image of 
$\hoz{j}{\pcong{A}{\mathfrak{p}}}\to \hoz{j}{H}$. 
Thus, comparing the $E^\infty$-terms of total degree $2$, we obtain
 a commutative diagram of the 
form 
\[
\xymatrix{
\hoz{2}{\pcong{A}{\mathfrak{p}}}\ar[r]\ar^-{\mathrm{id}}[d]
&\hoz{2}{\tcong{A}{\mathfrak{p}}}\ar[r]\ar^-{\iota_1}[d]
&C(\tcong{A}{\mathfrak{p}})\ar[r]\ar[d]
&0\\
\hoz{2}{\pcong{A}{\mathfrak{p}}}\ar[r]
&\hoz{2}{\spl{2}{A}}\ar[r]
&C(\spl{2}{A})\ar[r]
&0\\
}
\] 
where,  for $H=\tcong{A}{\mathfrak{p}}$ or $\spl{2}{A}$, 
 $C(H)$ is a group fitting into an exact sequence
\[
0\to E^{\infty}_{1,1}(H)\to C(H)\to E^\infty_{2,0}(H)\to 0.
\]

Since $\hoz{1}{\pcong{A}{\mathfrak{p}}}$ is a finite abelian $p$-group, it follows from 
Lemma \ref{lem:borel} that the natural maps 
\[
E^2_{i,1}(\tcong{A}{\mathfrak{p}})=\ho{i}{\bor(k)}{\hoz{1}{\pcong{A}{\mathfrak{p}}}}\to
\ho{i}{\spl{2}{k}}{\hoz{1}{\pcong{A}{\mathfrak{p}}}}=E^2_{i,1}(\spl{2}{A}) 
\]
are all isomorphisms. 

Since, furthermore, all these groups $E^2_{i,1}$ are finite abelian $p$-groups,
it follows that the differentials
\[
d^2_{i,0}:E^2_{i,0}(H)\to E^2_{i-1,1}(H)
\]
factor through $\psyl{E^2_{i,0}(H)}{p}$, for $H=\tcong{A}{\pi}$ or $\spl{2}{A}$.

By Lemma \ref{lem:borel} again, we have natural isomorphisms
\[
\psyl{E^2_{i,0}(\tcong{A}{\mathfrak{p}})}{p}=\psyl{\hoz{i}{\bor(k)}}{p}\cong 
\psyl{\hoz{i}{\spl{2}{k}}}{p}=\psyl{E^2_{i,0}(\spl{2}{A})}{p}.
\] 

Thus, we have 
\[
E^\infty_{1,1}(\tcong{A}{\mathfrak{p}})\cong E^\infty_{1,1}(\spl{2}{A})
\]
since 
\[
E^\infty_{1,1}(H)=\coker{d^2:\psyl{E^2_{3,0}(H)}{p}\to E^2_{1,1}(H)}
\]
when $H=\tcong{A}{\mathfrak{p}}$ or $\spl{2}{A}$. 

Finally, for $H=\tcong{A}{\mathfrak{p}}$ or $\spl{2}{A}$, we have 
\[
E^\infty_{2,0}(H)=\ker{d^2:E^2_{2,0}(H)\to E^2_{0,1}(H)}.
\]
But straightforward calculations (see, for example, \cite[Section 3]{hut:cplx13}) 
show that 
\[
E^2_{2,0}(\tcong{A}{\mathfrak{p}})=\hoz{2}{\bor(k)}= \psyl{\hoz{2}{\bor(k)}}{p}
= \psyl{\hoz{2}{\spl{2}{k}}}{p}= \hoz{2}{\spl{2}{k}}=E^2_{2,0}(\spl{2}{A}).
\]

Thus 
\[
E^\infty_{2,0}(\tcong{A}{\mathfrak{p}})\cong E^\infty_{2,0}(\spl{2}{A}).
\]

Hence the map $C(\tcong{A}{\mathfrak{p}})
\to C(\spl{2}{A})$ is an isomorphism, and the result follows. 
\end{proof}

\subsection{An exact sequence for the second homology of $\mathrm{SL}_2$ of 
$S$-integers}
Let $K$ be a global field and let $S\subset T$ be nonempty sets of primes of $K$ containing 
the infinite primes. Then there is a natural short exact sequence 
\[
\xymatrix{
0\ar[r]
& K_2(\ntr{S})\ar[r]
& K_2(\ntr{T})\ar^-{\sum\tau_{\mathfrak{p}}}[r]
& \oplus_{\mathfrak{p}\in T\setminus S}k(\mathfrak{p})^\times\ar[r]
&0.
}
\] 

In this section, we demonstrate an analogous exact sequence for $\hoz{2}{\spl{2}{\ntr{S}}}$, at
least when $S$ is sufficiently large.
 
\begin{thm}\label{thm:mv}
Let $A$ be a ring of $S$-integers in a global field $K$ where 
$|S|\geq 2$. 
Let $\mathfrak{p}$ be a nonzero prime ideal and let $p>0$ be the characteristic of the 
residue field $k(\mathfrak{p})$. 

Suppose also that  there exists $\lambda\in A^\times$ such 
that $\lambda^2-1\in A^\times$ also.

Let $x\in A$ and $m\geq 1$ such that$\mathfrak{p}^m=xA$. 

Then there is  a natural exact sequence 
\[
\xymatrix{
 \hoz{2}{\spl{2}{A}}\ar[r]
& \hoz{2}{\spl{2}{A[1/x]}}\ar^-{\delta}[r]
& \hoz{1}{k(\mathfrak{p})^\times}\ar[r]
& 0
}
\]

Here the map $\delta$ fits into a commutative diagram
\[
\xymatrix{
\hoz{2}{\spl{2}{A[1/x]}}\ar^-{\delta}[r]\ar[d]
& \hoz{1}{k(\mathfrak{p})^\times}\ar^-{\cong}[d]\\
\hoz{2}{\mathrm{SL}(K)}=\milk{2}{K}\ar^-{\tau_{\mathfrak{p}}}[r]
&k(\mathfrak{p})^\times\\
}
\] 
where $\tau_{\mathfrak{p}}:\milk{2}{K}\to k(\mathfrak{p})^\times$ is the tame symbol
\[
\tau_{\mathfrak{p}}(\{ x,y\})=(-1)^{v(x)v(y)}x^{-v(y)}y^{v(x)}\pmod{\mathfrak{p}}\in k(\mathfrak{p})^\times.
\]
\end{thm}
\begin{proof}

By Proposition \ref{prop:h2} the map 
$\iota_2:\hoz{2}{\tcong{A}{\mathfrak{p}}}\to \hoz{2}{H(\mathfrak{p})}$ is surjective. 

Thus the Mayer-Vietoris sequence yields the exact sequence
\[
\xymatrix{
 \hoz{2}{\spl{2}{A}}\ar[r]
& \hoz{2}{\spl{2}{A[1/x]}}\ar^-{\delta}[r]
& \hoz{1}{\tcong{A}{\mathfrak{p}}}\ar[r]
& 0.
}
\]

The remaining statements of the theorem follow from Proposition \ref{prop:delta} and 
Proposition \ref{prop:tcong}.
\end{proof}
\begin{rem} In this proof, 
the hypothesis that $\lambda,\lambda^2-1\in A^\times$ is needed so that Proposition \ref{prop:tcong}
is validly applied.

The first part of the proof only requires the weaker condition that 
$\lambda,\lambda^2-1\in A[1/x]^\times$. For example, taking $A=\Z[1/3]$, $\mathfrak{p}= 2A$,
$x=2$ 
and 
$\lambda=3$, 
we obtain an exact sequence
\[
\hoz{2}{\spl{2}{\Z[1/3]}}\to\hoz{2}{\spl{2}{\Z[1/6]}}\to \hoz{1}{\tcong{\Z[1/3]}{2}}\to 0.
\]

In this sequence, 
\[
\hoz{2}{\spl{2}{\Z[1/3]}}\cong \Z\mbox{ and } \hoz{2}{\spl{2}{\Z[1/6]}}\cong \Z\oplus \Z/2
\]
by the calculations of Adem-Naffah, \cite{adem:naffah}, and Tuan-Ellis, \cite{tuanellis}.

Thus 
\[
\hoz{1}{\tcong{\Z[1/3]}{2}}\not= 0
\]
while $\hoz{1}{k(2)^\times}=\hoz{1}{\F{2}^\times}=0$, so that the conclusion of Proposition 
\ref{prop:tcong} is false in the case $A=\Z[1/3]$ and $\mathfrak{p}=2A$. 
\end{rem}

\begin{rem} Theorem \ref{thm:mv} is not valid for more general Dedekind Domains $A$, 
even when there is a unit $\lambda$ such that $\lambda^2-1$ is also a unit.  

For example, let $k$ be an infinite field and let $K=k(t)$, $A=k[t]$, $\mathfrak{p}=tA$, 
$x=t$. It is shown in \cite[Theorem 4.1]{hut:laurent} that 
the cokernel of the natural map 
\[
\hoz{2}{\spl{2}{A}}\to\hoz{2}{\spl{2}{A[1/t]}}
\]   
is isomorphic to $\mwk{1}{k}$, the first Milnor-Witt $K$-group of the residue field $k$. It seems 
 reasonable to suppose that this statement should be true for a larger class of Dedekind Domains.

Note that there is a natural surjective map $\mwk{1}{k}\to \milk{1}{k}\cong k^\times$ which is an 
isomorphism when the field $k$ is finite. However, in general, the kernel of this homomorphism is 
$I^2(k)$ (see section \ref{sec:mwk} above). 
\end{rem}
\begin{cor}\label{cor:mv}
Let $K$ be a global field. 
Let $S$ be a set of primes of $K$ containing the infinite primes. 
Suppose 
that $\card{S}\geq 2$ and that $\ntr{S}$ contains a unit 
$\lambda$ such that $\lambda^2-1$ is also a unit.

Let $T$ be any set of primes containing $S$.
Then there is a natural exact sequence
\[
\hoz{2}{\spl{2}{\ntr{S}}}\to \hoz{2}{\spl{2}{\ntr{T}}}\to \oplus_{\mathfrak{p}\in T\setminus S}
k(\mathfrak{p})^\times\to 0.
\]
\end{cor}

\begin{proof}
We proceed by induction on $\card{T\setminus S}$. 

The case $\card{T\setminus S}=1$ is just Theorem \ref{thm:mv}. 

The inductive step follows immediately by applying the snake lemma to the commutative diagram 
with exact rows
\[
\xymatrix{
&\hoz{2}{\spl{2}{\ntr{T'}}}\ar[r]\ar@{>>}[d]
&\hoz{2}{\spl{2}{\ntr{T}}}\ar[r]\ar[d]
&k(\mathfrak{q})^\times\ar[r]\ar^-{\mathrm{id}}[d]
&0\\
0\ar[r]
&\oplus_{\mathfrak{p}\in T'}k(\mathfrak{p})^\times\ar[r]
&\oplus_{\mathfrak{p}\in T}k(\mathfrak{p})^\times\ar[r]
&k(\mathfrak{q})^\times\ar[r]
&0\\
}
\]
where $\mathfrak{q}$ is any  element in $T\setminus S$ and $T'=T\setminus\{ \mathfrak{q}\}$.  
\end{proof}

\begin{cor}\label{cor:mv2}
Let $K$ be a global field. Let $S$ be a set of primes of $K$ containing the infinite primes. Suppose 
that $\card{S}\geq 2$ and that $\ntr{S}$ contains a unit $\lambda$ such that $\lambda^2-1$ is also 
a unit.

Then there is a natural exact sequence
\[
\hoz{2}{\spl{2}{\ntr{S}}}\to \hoz{2}{\spl{2}{K}}\to 
\oplus_{\mathfrak{p}\not\in S}k(\mathfrak{p})^\times\to 0.
\]

\end{cor}

\begin{proof}
Since 
\[
K=\lim_{S\subset T}\ntr{T}
\]
this follows from Corollary \ref{cor:mv} by taking (co)limits. 
\end{proof}

\section{The main theorem}\label{sec:main}

Let $K$ be a global field. In this section, we use the results above to 
prove our main theorem which identifies 
$\hoz{2}{\spl{2}{\ntr{S}}}$ with a certain subgroup of $K_2(2,K)$, which we now describe.
 
For a prime  $\mathfrak{p}$ of $K$, we denote by $T_{\mathfrak{p}}$ the composite 
\[
\xymatrix{
K_2(2,K)\ar@{>>}[r]
&\milk{2}{K}\ar^-{\tau_{\mathfrak{p}}}[r]
&k(\mathfrak{p})^\times
}
\]
where $\tau_{\mathfrak{p}}$ is the tame symbol, as above. 

When $S$ is a nonempty set of primes of $K$ containing the infinite primes, we set
\[
\tilde{K}_2(2,\ntr{S}):=\ker{ K_2(2,K)\to \oplus_{\mathfrak{p}\not\in S}k(\mathfrak{p})^\times}.
\]

We begin by noting that this group is closely related to $K_2(\ntr{S})$:
\begin{lem}
For any global field $K$ and for any nonempty set $S$ of primes which contains the infinite 
primes there is a natural exact sequence
\[
0\to I^3(K)\to \tilde{K}_2(2,\ntr{S})\to K_2(\ntr{S})\to 0.
\]

In particular, $\tilde{K}_2(2,\ntr{S})\cong K_2(\ntr{S})$ if $K$ is of positive characteristic 
or is a totally imaginary number field. 
\end{lem}
\begin{proof}
Apply the snake lemma and Corollary \ref{cor:mwk} (1) to the map of short exact 
sequences
\[
\xymatrix{
0\ar[r]
&\tilde{K}_2(2,\ntr{S})\ar[r]\ar[d]
&K_2(2,K)\ar[r]\ar[d]
&\oplus_{\mathfrak{p}\not\in S}k(\mathfrak{p})^\times\ar^-{\mathrm{id}}[d]\ar[r]
&0\\
0\ar[r]
&K_2(\ntr{S})\ar[r]
&\milk{2}{K}\ar[r]
&\oplus_{\mathfrak{p}\not\in S}k(\mathfrak{p})^\times\ar[r]
&0.\\}
\] 

The second statement follows from the fact that, for a global field $K$,  
$I^3(K)\cong \Z^{r(K)}$ where 
$r(K)$ is the number of distinct real embeddings of $K$. 
\end{proof}

\begin{exa}\label{exa:z}
Consider the global field $K=\Q$. 

For any set $S$ of prime numbers, we will set
\[
\Z_S:=\Z[\{ 1/p\}_{p\in S}]=\ntr{S\cup \{\infty\}}.
\]

The kernel of the surjective map 
\[
K_2(2,\Q)\to \oplus_{p}\F{p}^\times
\]
is an infinite cyclic direct summand with generator $c(-1,-1)$. 

It follows that for any set $S$ of prime numbers  
\[
\tilde{K}_2(2,\Z_S)\cong \Z\oplus \left(\oplus_{p\in S}\F{p}^\times\right). 
\]
\end{exa}

More generally, we have the following description of the groups $\tilde{K}_2(2,\ntr{S})$:

For a global field $K$, let $\Omega$ be the set of real embeddings of $K$. For 
$\sigma\in \Omega$, there is a corresponding homomorphism 
\[
T_\sigma:K_2(K)\to \mu_2,\quad \{ a,b\}\mapsto 
\left\{
\begin{array}{ll}
-1, &\mbox{ if } \sgn{\sigma(a)},\sgn{\sigma(b)}<0\\
1,&\mbox{ otherwise }
\end{array}
\right.
\]

Let 
\[
 K_2(K)_+:=\ker{\oplus_{\sigma\in\Omega}T_\sigma:K_2(K)\to\mu_2^\Omega}
\]
and let $K_2(\ntr{S})_+=K_2(\ntr{S})\cap K_2(K)_+$.

\begin{lem}\label{lem:tilde}
Let $K$ be a global field. 
Let $S$ be a nonempty set of primes of $K$ including the infinite primes. 
Then 
\[
 \tilde{K}_2(2,\ntr{S})\cong K_2(\ntr{S})_+\oplus \Z^{\Omega}.
\]
\end{lem}

\begin{proof}
By classical quadratic form theory, the group $I^n(\R)$ is infinite cyclic with generator 
$\pf{-1}^n=(-2)^{n-1}\pf{-1}$.
 
It is shown in \cite{hut:laurent} that for a global field $K$ the natural surjective 
map 
\[
K_2(2,K)\to I^2(\R)^{\Omega}\cong \Z^{\Omega},\quad c(u,v)\mapsto 
(\pf{\sgn{\sigma(u)}}\pf{\sgn{\sigma(u)}})_{\sigma\in \Omega} 
\]
has kernel isomorphic to $K_2(K)_+$, where this isomorphism is realised by 
restricting the natural map $K_2(2,K)\to K_2(K)$. Furthermore, the composite map 
$I^3(K)\to K_2(2,K)\to I^2(\R)^{\Omega}$ induces an isomorphism 
\[
I^3(K)\cong I^3(\R)^{\Omega}=2\cdot(I^2(\R)^{\Omega}).
\]

Since $I^3(K)\subset \tilde{K}_2(2,\ntr{S})$,  the image of the map 
\[
\tilde{K}_2(2,\ntr{S})\to I^2(\R)^{\Omega}\cong \Z^{\Omega}
\]
contains a full sublattice.  

On the other hand, the kernel of this map is isomorphic -- via the map 
$K_2(2,K)\to K_2(K)$ -- to $K_2(\ntr{S})\cap K_2(K)_+$.  
\end{proof}

It is natural to ask, of course, about the relation between $\tilde{K}_2(2,\ntr{S})$ 
and $K_2(2,\ntr{S})$.

 It is a theorem of van der Kallen (\cite{vanderkallen:stab}) that when 
$K$ is a global field and when 
$S$ contains all infinite places and $\card{S}\geq 2$ then the stabilization map 
\[
K_2(2,\ntr{S})\to K_2(\ntr{S})
\]
is always surjective. 

We deduce:

\begin{lem} Let $K$ be a global field and let $S$ be a nonempty set of primes of $K$ containing 
the infinite primes.
Then the  image of the natural map $K_2(2,\ntr{S})\to K_2(2,K)\cong K_2(2,K)$ lies in 
$\tilde{K}_2(2,\ntr{S})$. 

Furthermore, when $\card{S}\geq 2$, and when there exist units 
$u_\sigma\in \ntr{S}^\times$, $\sigma\in \Omega$ satisfying 
\[
\sgn{\tau(u_\sigma)}=(-1)^{\delta_{\sigma,\tau}},
\]
the resulting natural map $K_2(2,\ntr{S})\to
\tilde{K}_2(2,\ntr{S})$ is surjective; i.e. the image of the map 
$K_2(2,\ntr{S})\to K_2(2,K)$ is precisely $\tilde{K}_2(2,\ntr{S})$.
\end{lem}
\begin{proof}
The diagram 
\[
\xymatrix{
&{K}_2(2,\ntr{S})\ar[r]\ar[d]
&K_2(2,K)\ar[r]\ar[d]
&\oplus_{\mathfrak{p}\not\in S}k(\mathfrak{p})^\times\ar^-{\mathrm{id}}[d]
&\\
0\ar[r]
&K_2(\ntr{S})\ar[r]
&\milk{2}{K}\ar[r]
&\oplus_{\mathfrak{p}\not\in S}k(\mathfrak{p})^\times\ar[r]
&0\\}
\]
commutes.

Our hypothesis on units ensures that the map 
\[
\tilde{K}_2(2,\ntr{S})\to K_2(2,K)\to I^2(\R)^{\Omega}
\]
is surjective.

Since we also have
\[
K_2(\ntr{S})_+\subset K_2(\ntr{S})\subset \tilde{K}_2(2,\ntr{S})
\]
by the result of van der Kallen, the second statement follows.

\end{proof}

One would expect that the resulting map $K_2(2,\ntr{S})\to \tilde{K}_2(2,\ntr{S})$ is very 
often an isomorphism. It seems to be difficult, however, to prove this in any given instance. 
In the case $K=\Q$, Jun Morita, \cite[Theorems 2,3]{morita:zs} has proved:

\begin{thm}\label{thm:morita}
Let $S$ be any of the following sets of primes numbers:

 $S=\{p_1,\ldots, p_n\}$,  the set of the first $n$ successive prime numbers, or
 $S$ is one of $\{2,5\}$, $\{ 2,3,7\}$, $\{ 2,3,11\}$, $\{ 2,3,5,11\}$, $\{ 2,3,13\}$, 
$\{ 2,3,7,13\}$, $\{2,3,17\}$, $\{ 2,3,5,19\}$.    

Then $K_2(2,\Z_S)$ is central in $\st{2}{\Z_S}$ and the natural map 
\[
K_2(2,\Z_S)\to \tilde{K}_2(2,\Z_S)\cong \Z\oplus \left(\oplus_{p\in S}\F{p}^\times\right)
\]
is an isomorphism. 
\end{thm}

\begin{lem}\label{lem:k2o}
Let $K$ be a global field and let 
$S$ be a nonempty set of primes of $K$ containing the infinite primes. Then the image of the 
natural map
\[
\xymatrix{
\hoz{2}{\spl{2}{\ntr{S}}}\to \hoz{2}{\spl{2}{K}}\ar^-{\cong}[r]
&K_2(2,K)
}
\]
lies in $\tilde{K}_2(2,\ntr{S})$.
\end{lem}

\begin{proof}
The diagram 
\[
\xymatrix{
\hoz{2}{\spl{2}{\ntr{S}}}\ar[r]\ar[d]
&\hoz{2}{\mathrm{SL}(\ntr{S})}\ar[d]\\
\hoz{2}{\spl{2}{K}}\ar[r]\ar^-{\cong}[d]
&\hoz{2}{\mathrm{SL}(K)}\ar^-{\cong}[d]\\
K_2(2,K)\ar[r]
&\milk{2}{K}\\
}
\]
commutes. But $\hoz{2}{\mathrm{SL}(\ntr{S})}\cong K_2(\ntr{S})$ and the natural map 
$K_2(\ntr{S})\to K_2(K)=\milk{2}{K}$ induces an isomorphism 
\[
K_2(\ntr{S})\cong\ker{\milk{2}{k}\to \oplus_{\mathfrak{p}\not\in S}k(\mathfrak{p})^\times}.
\]
\end{proof}

If $K$ is a global field and and if $S$ is a nonempty set of primes containing the infinite primes
we will let 
\[
\mathcal{K}_S:=\ker{\hoz{2}{\spl{2}{\ntr{S}}}\to \hoz{2}{\spl{2}{K}}}.
\]

Note that 
\[
\mathcal{K}_S:=\ker{\hoz{2}{\spl{2}{\ntr{S}}}\to \tilde{K}_2(2,\ntr{S})}
\]
since $\tilde{K}_2(2,\ntr{S})\subset K_2(2,K)\cong \hoz{2}{\spl{2}{K}}$. 

\begin{rem}
In general, the kernels $\mathcal{K}_S$ can be arbitrarily large, even in the case $K=\Q$: 

The calculations of Adem-Naffah, \cite{adem:naffah}, show that the ranks of the groups 
$\hoz{2}{\spl{2}{\Z[1/p]}}$ grow with linearly $p$ when $p$ is a prime number.  On the other 
hand, the rank of $\hoz{2}{\spl{2}{\Q}}$ is $1$. 
\end{rem}

\begin{lem}\label{lem:h2o}
Let $K$ be a global field. Let $S$ be a set of primes of $K$ containing the infinite primes. Suppose 
that $\card{S}\geq 2$ and that $\ntr{S}$ contains a unit $\lambda$ such that $\lambda^2-1$ is also a unit.

Then 
\begin{enumerate}
\item The natural map 
\[
\hoz{2}{\spl{2}{\ntr{S}}}\to\tilde{K}_2(2,\ntr{S})
\]
is surjective.
\item If $T\supset S$, then the natural map 
$\mathcal{K}_S\to\mathcal{K}_T$ 
is surjective.
\end{enumerate}
\end{lem}

\begin{proof}\

\begin{enumerate}
\item By Corollary \ref{cor:mv2}, we have a commutative diagram with exact rows
\[
\xymatrix{
&\hoz{2}{\spl{2}{\ntr{S}}}\ar[r]\ar[d]
&\hoz{2}{\spl{2}{K}}\ar[r]\ar^-{\cong}[d]
&\oplus_{\mathfrak{p}\not\in S}k(\mathfrak{p})^\times\ar[r]\ar^-{\mathrm{id}}[d]
&0\\
0\ar[r]
&\tilde{K}_2(2,\ntr{S})\ar[r]
&K_2(2,K)\ar[r]
&\oplus_{\mathfrak{p}\not\in S}k(\mathfrak{p})^\times\ar[r]
&0.\\
}
\]
\item
Apply the snake lemma to the diagram
\[
\xymatrix{
&\hoz{2}{\spl{2}{\ntr{S}}}\ar[r]\ar[d]
&\hoz{2}{\spl{2}{\ntr{T}}}\ar[r]\ar[d]
&\oplus_{\mathfrak{p}\in T\setminus S}k(\mathfrak{p})^\times\ar[r]\ar^-{\mathrm{id}}[d]
&0\\
0\ar[r]
&\tilde{K}_2(2,\ntr{S})\ar[r]
&\tilde{K}_2(2,\ntr{T})\ar[r]
&\oplus_{\mathfrak{p}\in T\setminus S}k(\mathfrak{p})^\times\ar[r]
&0.\\
}
\]
\end{enumerate}
\end{proof}

\begin{rem} 
Note, on the other hand, that the map 
\[
0=\hoz{2}{\spl{2}{\Z}} \to \tilde{K}_2(2,\Z)\cong K_2(2,\Z)\cong \Z
\]
cannot be surjective.
\end{rem}

\begin{thm}\label{thm:s}
 Let $K$ be a global field. 
\begin{enumerate}
\item There exists a finite set  $S$ of primes of $K$ satisfying 
\begin{enumerate}
\item $S$ contains all infinite primes and $\card{S}\geq 2$.
\item There exists a unit $\lambda$ of $\ntr{S}$ such that $\lambda^2-1$ is also a unit.
\item The natural map $\hoz{2}{\spl{2}{\ntr{S}}}\to \tilde{K}_2(2,\ntr{S})$ is an isomorphism. 
\end{enumerate}

\item If $T$ is any set of primes containing $S$ then 
$\hoz{2}{\spl{2}{\ntr{T}}}\cong \tilde{K}_2(2,\ntr{T})$; i.e. there is a natural short 
exact sequence
\[
\xymatrix{
0\ar[r]
& \hoz{2}{\spl{2}{\ntr{T}}}\ar[r]
& K_2(2,K)\ar^-{\sum{T_{\mathfrak{p}}}}[r] 
&\oplus_{\mathfrak{p}\not\in T}k(\mathfrak{p})^\times\ar[r]
& 0.\\
}
\] 
\end{enumerate}
\end{thm}
\begin{proof}\ 

\begin{enumerate}
\item Let $S_0$ be any set of primes satisfying (a) and (b). Since $\hoz{2}{\spl{2}{\ntr{S}}}$
is a finitely-generated abelian group, so also is  $\mathcal{K}_{S_0}$. Since 
$\hoz{2}{\spl{2}{K}}=\lim_T\hoz{2}{\spl{2}{\ntr{T}}}$, the limit being taken over finite sets $T$ 
of primes, it follows that there is a finite 
set of primes $S$ containing $S_0$ for which 
\[
\mathcal{K}_{S_0}=\ker{\hoz{2}{\spl{2}{\ntr{S_0}}}\to\hoz{2}{\spl{2}{\ntr{S}}}}. 
\]

By Lemma \ref{lem:h2o}  it follows that $\mathcal{K}_S=0$ and hence that 
$\hoz{2}{\spl{2}{\ntr{S}}}\cong \tilde{K}_2(2,\ntr{S})$ as required. 
\item This is immediate from Lemma \ref{lem:h2o}.
\end{enumerate}
\end{proof}

\begin{lem}\label{lem:z}
Let $K=\Q$ and let $S=\{ 2,3,\infty\}$. Then $S$ satisfies conditions (a)-(c) of 
Theorem \ref{thm:s} (1).
\end{lem}
\begin{proof}
The set $S=\{ 2,3,\infty\}$ clearly satisfies conditions (a) and (b). 

Observe that 
\[
\ntr{S}=\Z_{\{ 2,3\}}=\Z\left[ \frac{1}{2},\frac{1}{3}\right]=\Z\left[\frac{1}{6}\right].
\]

By Lemma \ref{lem:h2o}, the natural map 
\[
\hoz{2}{\spl{2}{\Z[1/6]}}\to \tilde{K}_2(2,\Z[1/6])
\cong \Z\oplus \F{3}^\times
\]
(see Example \ref{exa:z}) is surjective.

On the other hand, the calculations of Tuan and Ellis, \cite{tuanellis},  show that
\[
\hoz{2}{\spl{2}{\Z[1/6]}}\cong \Z\oplus \Z/2\Z.
\]

It follows that the natural map above is an isomorphism. 
\end{proof}

In view of Theorem\ref{thm:s} (2) and Example \ref{exa:z}, we immediately deduce:
\begin{thm}\label{thm:z}
Let $T$ be any set of prime numbers containing $2, 3$. Then there is an isomorphism
\[
\hoz{2}{\spl{2}{\Z_T}}\cong \Z\oplus\left(\oplus_{p\in T}\F{p}^\times\right).
\]

In particular, if $m\in \Z$ and if $6|m$ then 
\[
\hoz{2}{\spl{2}{\Z[1/m]}}\cong \Z\oplus\left(\oplus_{p|m}\F{p}^\times\right).
\]
\end{thm}

Combining this with Morita's Theorem (\ref{thm:morita}) we deduce:

\begin{prop}\label{prop:morita}
Let $S$ be any of the following sets of primes numbers:

 $S=\{p_1,\ldots, p_n\}$,  the set of the first $n$ successive prime numbers, or
 $S$ is one of $\{ 2,3,7\}$, $\{ 2,3,11\}$, $\{ 2,3,5,11\}$, $\{ 2,3,13\}$, 
$\{ 2,3,7,13\}$, $\{2,3,17\}$, $\{ 2,3,5,19\}$.    
 
Then the natural map  
\[
\hoz{2}{\spl{2}{\Z_{S}}}\to K_2(2,\Z_{S})
\]
is an isomorphism and 
\[
1\to K_2(2,\Z_{S})\to \st{2}{\Z_{S}}\to \spl{2}{\Z_{S}}\to 1 
\]
is the universal central extension of $\spl{2}{\Z_{S}}$. 
\end{prop}
\begin{proof}
Since $K_2(2,\Z_S)$ is central in $\st{2}{\Z_S}$, there is a natural homomorphism 
$\hoz{2}{\spl{2}{\Z_{S}}}\to K_2(2,\Z_{S})$ through which the map 
$\hoz{2}{\spl{2}{\Z_{S}}}\to K_2(2,\Q)$ factors. 

Since $\hoz{2}{\spl{2}{\Z_{S}}}$ and $K_2(2,\Z_{S})$ are both isomorphic to 
$\tilde{K}_2(2,\Z_S)\subset K_2(2,\Q)$, the result follows immediately.
\end{proof}

\section{Some $2$-dimensional  homology classes}\label{sec:classes}

In this section we construct  explicit cycles in the bar resolution of $\spl{2}{A}$ which 
represent homology classes in $\hoz{2}{\spl{2}{A}}$. We show that these classes map
to the symbols $c(u,v)\in K_2(2,A)$,  when $A$ is a field.

\subsection{The homology classes $C(a,b)$}
Let $A$ be a commutative ring and let $a\in A^\times$. We define the following elements of 
$\spl{2}{A}$:
\[
w:=\matr{0}{1}{-1}{0},\quad G_a:=\matr{0}{-1}{1}{a+a^{-1}},\quad 
H_a:= E_{21}(a)=\matr{1}{0}{a}{1}.
\]  

Note that 
\[
wG_a= \matr{1}{a+a^{-1}}{0}{1}=E_{12}(a+a^{-1}).
\]
We also define
\[
R_a:=H_aG_aH^{-1}_a=H_{a}G_aH_{-a}=\matr{a}{-1}{0}{a^{-1}}.
\]
Thus, by definition, 
\[
H_aG_a=R_aH_a=\matr{0}{-1}{1}{a^{-1}}.
\]

Let 
\[
\Theta_a:= [H_a|G_a]-[R_a|H_a]+[w^{-1}|wG_a]\in \bar{F}_2(\spl{2}{A})=\bar{F}_2.
\]

Then 
\[
d_2(\Theta_a)= [w^{-1}]+[wG_a]-[R_a]
%=[w^{-1}]+[E_{12}(-(a+a^{-1}))]-\left[ \matr{a}{-1}{0}{a^{-1}}\right]
\in \bar{F}_1.
\]

Now let $a,b\in A^\times$. Then 
\begin{eqnarray*}
d_2(\Theta_{ab}-\Theta_a-\Theta_b+\Theta_1)=
\left( [R_a]+[R_b]-[R_{ab}]\right)+\left( [wG_{ab}]-[wG_a]-[wG_b]+[wG_1]-[R_1]\right).
\end{eqnarray*}

Now 
\[
[R_a]+[R_b]=[R_aR_b]+d_2\left([R_a|R_b]\right)
\]
and 
\[
[R_{ab}]= [R_aR_b]+[(R_aR_b)^{-1}]-d_2\left([R_aR_b|(R_aR_b)^{-1}(R_{ab})]\right).
\]

Hence
\[
[R_a]+[R_b]-[R_{ab}]=-[(R_aR_b)^{-1}R_{ab}]+d_2\left([R_a|R_b]+[R_aR_b|(R_aR_b)^{-1}(R_{ab})]\right)
\]
where
\[
(R_aR_b)^{-1}R_{ab}=\matr{1}{(ab)^{-1}(a+b^{-1}-1)}{0}{1}=E_{12}((ab)^{-1}(a+b^{-1}-1)).
\]

Putting this together, we deduce 
\begin{eqnarray*}
&&d_2(\Theta_{ab}-\Theta_a-\Theta_b+\Theta_1 - [R_a|R_b]-[R_aR_b|(R_aR_b)^{-1}(R_{ab})])\\
&=& [wG_{ab}]-[wG_a]-[wG_b]+[wG_1]-[R_1]+[(R_aR_b)^{-1}(R_{ab})]\\
&=& [E_{12}(ab+(ab)^{-1})]-[E_{12}(a+a^{-1})]-[E_{12}(b+b^{-1})]\\
&&+[E_{12}(2)]-[E_{12}(-1)]+
[E_{12}((ab)^{-1}(a+b^{-1}-1))].
\end{eqnarray*}

Now suppose that there exists $\lambda\in A^\times$ such that $\lambda^2-1\in A^\times$. 
Let
\[
D(\lambda):=\matr{\lambda}{0}{0}{\lambda^{-1}}\in\spl{2}{A}.
\]

Recall that for any $x\in A$ 
\[
D(\lambda)E_{12}(x)D(\lambda)^{-1}= E_{12}(\lambda^2x)
\]
and hence for any $x\in A$ we have  
\begin{eqnarray*}
E_{12}(x)&=& D(\lambda)E_{12}(x')D(\lambda)^{-1}E_{12}(x')^{-1}\\
&=&D(\lambda)E_{12}(x')(E_{12}(x')D(\lambda))^{-1}= [D(\lambda),E_{12}(x')].
\end{eqnarray*}
where $x':= x/(\lambda^2-1)$.

Now if $G$ is any group and if $g,h\in G$ then 
\[
D_2([(gh)(hg)^{-1}|hg]-[g|h]+[h|g])= [(gh)(hg)^{-1}]=\left[ [g,h]\right].
\]

Thus, we define 
\[
\Psi_x=\Psi_{x,\lambda}:=\left[E_{12}(x)|E_{12}(x')D(\lambda)\right]-\left[ D(\lambda)| E_{12}(x')\right]
+\left[ E_{12}(x')|D(\lambda)\right]\in \bar{F}_2.
\]
By the preceeding remarks, we have $d_2(\Psi_{x,\lambda})=[E_{12}(x)]\in \bar{F}_1$ for any 
$x\in A$. 

From all of these calculations we deduce: 
\begin{prop}\label{prop:fab}
Let $A$ be a commutative ring. Suppose that there exists $\lambda\in A^\times$ such that 
$\lambda^2-1\in A^\times$. Let $a,b\in A^\times$. Then 
\begin{eqnarray*}
F(a,b)_\lambda:&=& [R_a|R_b]+[R_aR_b|(R_aR_b)^{-1}(R_{ab})]+\Theta_a+\Theta_b-\Theta_{ab}-\Theta_1 \\
&+&\Psi_{ab+(ab)^{-1}}-\Psi_{a+a^{-1}}-\Psi_{b+b^{-1}}+\Psi_{2}-\Psi_{-1}+\Psi_{(ab)^{-1}(a+b^{-1}-1)}\in \bar{F}_2 
\end{eqnarray*}
is a cycle, representing a homology class $C(a,b)_{\lambda}\in \hoz{2}{\spl{2}{A}}$.
\end{prop}

\begin{rem}
The cycles $F(a,b)_{\lambda}$ are clearly functorial in the sense that if $\psi:A\to B$ is a homomorphism
of commutative rings and if $a,b,\lambda,\lambda^2-1\in A^\times$ then 
\[
\psi_*(F(a,b)_{\lambda})=F(\psi(a),\psi(b))_{\psi(\lambda)}\in \bar{F}_2(\spl{2}{B}).
\]
\end{rem}

\begin{rem}
More generally, suppose that 
$\Lambda=(\lambda_1,\ldots, \lambda_n,b_1,\ldots, b_n)\in (A^\times)^n\times (A^n)$
satisfies
\[
\sum_{i=1}^n(\lambda_i^2-1)b_i=1
\]

Then for any $x\in A$
\[
E_{12}(x)=
\prod_i[D(\lambda_i),E_{12}(b_ix)].
\]
by the proof of Proposition \ref{prop:e2}.

Since
\[
[\prod_{i=1}^nc_i]=\sum_{i=1}^n[c_i]-d_2\left(\sum_{k=1}^{n-1}[c_1\cdots c_k|c_{k+1}]\right)
\]
in $\bar{F}_1(A)$, we can easily write down an element $\Psi_{x,\Lambda}\in \bar{F}_2(A)$ satisfying 
$d_2(\Psi_{x,\Lambda})=[E_{12}(x)]$ and thus construct cycles $F(a,b)_{\Lambda}$. 
\end{rem}

\begin{rem} Specialising to the case $a=b=-1$ we obtain:
\[
F(-1,-1)_{\lambda}=[R_{-1}|R_{-1}]+[E_{12}(2)|E_{12}(-3)]+\Psi_{-3}-\Psi_{-1}\\
+2( \Theta_{-1}-\Theta_1+\Psi_{2}-\Psi_{-2}).
\]
\end{rem}

As we will see, when $A$ is a field with at least four elements, 
the homology class $C(a,b)_{\lambda}$ does not depend on the choice of $\lambda$. In fact, this is the case 
for many commutative rings. For example, we have: 

\begin{lem}\label{lem:indep}
Let $A$ be a commutative ring. Suppose there exists $n\in \Z$ such that $n,n^4-1\in A^\times$. 

Then, for any $a,b\in A^\times$, 
the homology class $C(a,b)_{\lambda}\in \hoz{2}{\spl{2}{A}}$ is independent of the choice of 
$\lambda$. 
\end{lem}
\begin{proof}
 Suppose that $\lambda,\mu\in A^\times$ satisfy the condition $\lambda^2-1,\mu^2-1
\in A^\times$.

Let $a,b\in A^\times$. Note that $F(a,b)_{\lambda}-F(a,b)_\mu$ is a sum or difference of terms of the 
form $\Psi_{x,\lambda}-\Psi_{x,\mu}$, $x\in A$. We will show that each such term is a boundary. 

We begin by observing that, for any $x\in A$, the elements $\Psi_{x,\lambda}$ and $\Psi_{x,\mu}$ lie in 
$\bar{F}_2(\bor)$ where 
\[
\bor:= \left\{\matr{u}{y}{0}{u^{-1}}\in\spl{2}{A}\ |\ u\in A^\times\right\}
\]    
is the subgroup of upper-triangular matrices in $\spl{2}{A}$. 

Note that 
\[
d_2(\Psi_{x,\lambda}-\Psi_{x,\mu})=[E_{12}(x)]-[E_{12}(x)]=0
\]
so that $\Psi_{x,\lambda}-\Psi_{x,\mu}$ represents a homology class in $\hoz{2}{\bor}$. We will show that 
it represents the trivial class. 

Let $T:= \{ D(u)\ |\ u\in A^\times \}$ be the group of diagonal matrices and let 
$\pi:\bor\to T$ be the natural surjective homomorphism sending $D(u)E_{12}(z)$ to $D(u)$. 
Then
\[
U:=\ker{\pi}=\{ E_{12}(y)\ |\ y\in A\}
\]
is the group of unipotent matrices. 

We have $T\cong A^\times$ via $D(u)\leftrightarrow u$ and 
$U\cong A$, via $E_{12}(x)\leftrightarrow x$. 

Note that 
\begin{eqnarray*}
\pi(\Psi_{x,\lambda})&=& \pi\left( \left[E_{12}(x)|E_{12}(x')D(\lambda)\right]-
\left[ D(\lambda)| E_{12}(x')\right]\right)\\
&=& [I|D(\lambda)]-[D(\lambda)|I]+[I|D(\lambda)]\in \bar{F}_2(T).
\end{eqnarray*}
Since 
\[
d_3([X|I|I]=[I|I]-[I|X]\mbox{ and } d_3([I|I|X])=[X|I]-[I|I]
\]
it follows easily that $\pi(\Psi_{x,\lambda}-\Psi_{x,\mu})\in d_3(\bar{F}_3(T))$. 
Thus $\pi(\Psi_{x,\lambda}-\Psi_{x,\mu})$ represents the trivial homology class in $\hoz{2}{T}$. 

To conclude, we will show that our hypotheses are enough to ensure that $\pi$ induces an isomorphism 
$\hoz{2}{\bor}\cong \hoz{2}{T}$. 

We consider the Hochschild-Serre spectral sequence associated to the short exact sequence 
\[
1\to U\to \bor \to T\to 1.
\]

This has the form
\[
E^2_{i,j}=\ho{i}{T}{\hoz{j}{U}}\Rightarrow \hoz{i+j}{\bor}
\]
$D(u)\in T$ acts by conjugation on $U\cong A$ as multiplication by $u^2$. Thus the induced 
action of $D(u)$ on $\hoz{2}{U}\cong U\bigwedge_{\Z}U\cong A\bigwedge_{\Z}A$ is multiplication 
by $u^4$. 

In particular, $D(n)$ acts as multiplication by $n^2$ on $\hoz{1}{U}$, and as multiplication 
by $n^4$ on $\hoz{2}{U}$.

Since $T$ is abelian, and since $n^2-1$, $n^4-1$ are units in $A$, 
it follows from the  ``centre kills'' argument that $\ho{i}{T}{\hoz{j}{U}}=0$ 
for $1\leq j\leq 2$.

Thus, from the spectral sequence, the map $\pi$ induces an isomorphism 
$\hoz{n}{\bor}\cong\hoz{n}{T}$ for $n\leq 2$.  
\end{proof}
\begin{rem} Since $2^4-1=3\cdot 5$, the condition of the Lemma \ref{lem:indep} is satisfied 
by any ring in which 
$2,3$ and $5$ are units.
\end{rem}

\subsection{A variation} We describe a slightly more compact $2$-cycle $\tilde{F}(a,b)_{\lambda}$ 
in the case where $a^2-1, b^2-1$ and $(ab)^2-1$ are all units in $A$. 

Suppose that $a\in A$ is a unit such that $a^2-1\in A^\times$ also. 
Let
\[
\tilde{H}_a=\matr{\frac{1}{1-a}}{\frac{a}{1-a}}{\frac{a}{1+a}}{\frac{1}{1+a}}\in \spl{2}{A}.
\] 

Then 
\[
\tilde{H}_aG_a\tilde{H}_a^{-1}=\matr{a}{0}{0}{a^{-1}}=D(a).
\]

Thus if we let 
\[
\tilde{\Theta}_a:=[\tilde{H}_a|G_a]-[D(a)|\tilde{H}_a]+[w^{-1}|wG_a]\in \bar{F}_2
\]
then
\[
d_2(\tilde{\Theta}_a)=[w^{-1}]+[wG_a]-[D(a)].
\]

If $a^2-1, b^2-1, (ab)^2-1\in A^\times$ then 
\begin{eqnarray*}
d_2(\tilde{\Theta}_a+\tilde{\Theta}_b-\tilde{\Theta}_{ab}-\Theta_1)&=& [D(ab)]-[D(a)]-[D(b)]\\
&&+[E_{12}(a+a^{-1})]+[E_{12}(b+b^{-1})]-[E_{12}(ab+(ab)^{-1})]+[E_{12}(-1)]-[E_{12}(2)]\\
&=& d_2(-[D(a)|D(b)]+\Psi_{a+a^{-1}}+\Psi_{b+b^{-1}}-\Psi_{ab+(ab)^{-1}}+\Psi_{-1}-\Psi_2).\\
\end{eqnarray*}

We deduce:
\begin{prop}\label{prop:var}
If $a,b,\lambda,a^2-1,b^2-1,(ab)^2-1,\lambda^2-1\in A^\times$ then 
\begin{eqnarray*}
\tilde{F}(a,b)_\lambda:= [D(a)|D(b)]+\tilde{\Theta}_a+\tilde{\Theta}_b-\tilde{\Theta}_{ab}-\Theta_1 
+\Psi_{ab+(ab)^{-1}}-\Psi_{a+a^{-1}}-\Psi_{b+b^{-1}}+\Psi_{2}-\Psi_{-1}
\end{eqnarray*}
is a cycle representing a homology class $\tilde{C}(a,b)_\lambda\in \hoz{2}{\spl{2}{A}}$. 
\end{prop}

\subsection{Symbols as homology classes}
In this section, the map of sets  $s:\spl{2}{F}\to\st{2}{F}$ and the homomorphism 
 $\bar{f}:\hoz{2}{\spl{2}{F}}\to K_2(2,F)$ are those described in section \ref{sec:map} above.

\begin{thm}\label{thm:symb}
 Let $F$ be a field with at least four elements. 
Let $\lambda\in F^\times\setminus \{ \pm 1\}$. 
\begin{enumerate}
\item Let $a,b\in F^\times$. Then 
\[
\bar{f}(C(a,b)_{\lambda})=c(a,b).
\] 
\item Suppose further that $a,b,ab\not\in \{ \pm 1\}$. Then
\[
\bar{f}(\tilde{C}(a,b)_{\lambda})=c(a,b).
\] 
\end{enumerate}
\end{thm}

\begin{proof}
We begin by noting that, by Lemma \ref{lem:f}, we have 
\[
\bar{f}(\Psi_x)=1 \forall x\in F\mbox{ and } \bar{f}([R_aR_b|(R_aR_b)^{-1}(R_{ab})])=1
\]
 since $c(1,v)=c(u,1)=1$ in $K_2(2,F)$.  

Also, by Lemma \ref{lem:f}, 
\[
\bar{f}([R_a|R_b])=\bar{f}([D(a)|D(b)]=c(a,b).
\]

\begin{enumerate}
\item
For any $u\in F^\times$  
\begin{eqnarray*}
\bar{f}([H_u|G_u])&=& s(H_u)s(G_u)s(H_uG_u)^{-1}\\
&=& x_{21}(u)\cdot w_{12}(-1)x_{12}(u+u^{-1})\cdot x_{12}(-u^{-1})w_{12}(1)\\
&=&x_{21}(u)\cdot (w_{12}(-1)x_{12}(u)w_{12}(1))\\
&=&x_{21}(u)x_{12}(u)^{w_{12}(1)}\\
&=&x_{21}(u)x_{21}(-u)=1\mbox{\quad by Lemma \ref{lem:conjw}}.\\ 
\end{eqnarray*}
and
\begin{eqnarray*}
\bar{f}([w^{-1}|wG_u])&=& s(w^{-1})s(wG_u)s(G_u)^{-1}\\
&=& w_{12}(-1)x_{12}(u+u^{-1})\cdot\left(w_{12}(-1)x_{12}(u+u^{-1})\right)^{-1}=1.
\end{eqnarray*}

Furthermore
\begin{eqnarray*}
\bar{f}([R_u|H_u])&=& s(R_u)s(H_u)s(R_uH_u)^{-1}\\
&=&x_{12}(-u)h_{12}(u)x_{21}(u)x_{12}(-u^{-1})w_{12}(1)\\
&=&h_{12}(u)x_{12}(-u^{-1}) x_{21}(u)x_{12}(-u^{-1})w_{12}(1)\mbox{ since $x_{12}(-u)^{h_{12}(u)}
=x_{12}(-u^{-1})$}\\
&=& h_{12}(u)w_{12}(-u^{-1})w_{12}(1).\\
 \end{eqnarray*}

Now 
\[
w_{12}(-u^{-1})w_{12}(1)=w_{12}(-u^{-1})w_{12}(-1)w_{12}(-1)^{-1}w_{12}(-1)^{-1}
=h_{12}(-u^{-1})h_{12}(-1)^{-1}
\]
and hence
\[
\bar{f}([R_u|H_u])=c(u,-u^{-1})=c(-u,u)=1.
\]

Thus 
\[
\bar{f}(\Theta_u)=1
\]
for all units $u$. 

Putting all of this together gives $\bar{f}(F(a,b)_\lambda)=c(a,b)$ as required.

\item We must show that $\bar{f}(\tilde{\Theta}_a)=1$ whenever $a,a^2-1\in F^\times$. 

As above, we have $\bar{f}([w^{-1}|wG_a])=1$. 

Now, 
\begin{eqnarray*}
s(D(a))=h_{12}(a),\quad  s(\tilde{H}_a)=x_{12}\left(\frac{1+a}{a(1-a)}\right)
w_{12}\left(\frac{1+a}{-a}\right)x_{12}(a^{-1}),\\
\mbox{ and }\quad 
s(D(a)\tilde{H}_a)=x_{12}\left(\frac{a(1+a)}{1-a}\right)w_{12}(-(1+a))x_{12}(a^{-1}).
\end{eqnarray*}
Thus
\begin{eqnarray*}
\bar{f}([D(a)|\tilde{H}_a])&=& s(D(a))s(\tilde{H}_a)s(D(a)\tilde{H}_a)^{-1}\\
&=&h_{12}(a)x_{12}\left(\frac{1+a}{a(1-a)}\right)
w_{12}\left(\frac{1+a}{-a}\right)x_{12}(a^{-1})x_{12}(-a^{-1})w_{12}(1+a)
x_{12}\left(\frac{a(1+a)}{a-1}\right)\\
&=&h_{12}(a)x_{12}\left(\frac{1+a}{a(1-a)}\right)
w_{12}\left(\frac{1+a}{-a}\right)w_{12}(1+a)
x_{12}\left(\frac{a(1+a)}{a-1}\right)\\
&=&h_{12}(a)w_{12}\left(\frac{1+a}{-a}\right)x_{21}\left(\frac{-a}{1-a^2}\right)w_{12}(1+a)
x_{12}\left(\frac{a(1+a)}{a-1}\right)\\
\end{eqnarray*}
using
\[
x_{12}\left(\frac{1+a}{a(1-a)}\right)^{
w_{12}\left(\frac{1+a}{-a}\right)}= x_{21}\left(\frac{-a}{1-a^2}\right).
\]

Since, by Lemma \ref{lem:conjw},
\[
x_{21}\left(\frac{-a}{1-a^2}\right)^{w_{12}(1+a)}= x_{12}\left(\frac{a(1+a)}{1-a}\right),
\]
this gives 
\begin{eqnarray*}
\bar{f}([D(a)|\tilde{H}_a])&=&h_{12}(a)
w_{12}\left(\frac{1+a}{-a}\right)w_{12}(1+a)
x_{12}\left(\frac{a(1+a)}{1-a}\right)x_{12}\left(\frac{a(1+a)}{a-1}\right)\\
&=&h_{12}(a)
w_{12}\left(\frac{1+a}{-a}\right)w_{12}(1+a)\\
&=&h_{12}(a)h_{12}\left(\frac{1+a}{-a}\right)h_{12}(-(1+a))^{-1}\\
&=&c(a,-(1+a)a^{-1})=c(a,1+a).
\end{eqnarray*}

Now 
\[
s(\tilde{H}_aG_a)=s\left(\matr{\frac{a}{1-a}}{\frac{a^2}{1-a}}{\frac{1}{1+a}}{\frac{a^{-1}}
{1+a}}\right)= x_{12}\left(\frac{a(1+a)}{1-a}\right)w_{12}(-(1+a))x_{12}(a^{-1}).
\]

So
\begin{eqnarray*}
\bar{f}([\tilde{H}_a|G_a])&=& s(\tilde{H}_a)s(G_a)s(\tilde{H}_aG_a)^{-1}\\
&=&x_{12}\left(\frac{1+a}{a(1-a)}\right)
w_{12}\left(\frac{1+a}{-a}\right)x_{12}(a^{-1})w_{12}(-1)x_{12}(a+a^{-1})
x_{12}(-a^{-1})w_{12}(1+a)
x_{12}\left(\frac{a(1+a)}{a-1}\right)\\
&=&x_{12}\left(\frac{1+a}{a(1-a)}\right)
w_{12}\left(\frac{1+a}{-a}\right)x_{12}(a^{-1})w_{12}(-1)x_{12}(a)w_{12}(1+a)
x_{12}\left(\frac{a(1+a)}{a-1}\right)\\ 
&=& w_{12}\left(\frac{1+a}{-a}\right)x_{21}\left(\frac{-a}{1-a^2}\right)w_{12}(-1)
x_{21}(-a^{-1})x_{12}(a)w_{12}(1+a)
x_{12}\left(\frac{a(1+a)}{a-1}\right)\\ 
\end{eqnarray*}
using
\[
x_{12}\left(\frac{1+a}{a(1-a)}\right)^{w_{12}\left(\frac{1+a}{-a}\right)}= x_{21}\left(\frac{-a}{1-a^2}\right)
\mbox{ and }
x_{12}(a^{-1})^{w_{12}(-1)}=x_{21}(-a^{-1}).
\]

Since $x_{12}(-a)w_{12}(a)=x_{21}(-a^{-1})x_{12}(a)$, we thus have

\begin{eqnarray*}
\bar{f}([\tilde{H}_a|G_a])
&=& w_{12}\left(\frac{1+a}{-a}\right)x_{21}\left(\frac{-a}{1-a^2}\right)w_{12}(-1)
x_{12}(-a)w_{12}(a)w_{12}(1+a)
x_{12}\left(\frac{a(1+a)}{a-1}\right)\\ 
&=& w_{12}\left(\frac{1+a}{-a}\right)w_{12}(-1)x_{12}\left(\frac{a}{1-a^2}\right)
x_{12}(-a)w_{12}(a)w_{12}(1+a)
x_{12}\left(\frac{a(1+a)}{a-1}\right).\\
\end{eqnarray*}

Since 
\[
x_{12}\left(\frac{a}{1-a^2}\right)
x_{12}(-a)=x_{12}\left(\frac{a}{1-a^2}-a\right)= x_{12}\left(\frac{a^3}{1-a^2}\right),
\]
we obtain
\begin{eqnarray*} 
\bar{f}([\tilde{H}_a|G_a])
&=& w_{12}\left(\frac{1+a}{-a}\right)w_{12}(-1)
x_{12}\left(\frac{a^3}{1-a^2}\right)w_{12}(a)w_{12}(1+a)
x_{12}\left(\frac{a(1+a)}{a-1}\right).\\
\end{eqnarray*}

The conjugation rules of Corollary \ref{cor:conjw} give 
\[
x_{12}\left(\frac{a^3}{1-a^2}\right)w_{12}(a)w_{12}(1+a)=
w_{12}(a)x_{12}\left(\frac{-a}{1-a^2}\right)w_{12}(1+a)=
w_{12}(a)w_{12}(1+a)x_{12}\left(\frac{a(1+a)}{1-a}\right).
\]

Therefore
\begin{eqnarray*}
\bar{f}([\tilde{H}_a|G_a])
&=& w_{12}\left(\frac{1+a}{-a}\right)w_{12}(-1)
w_{12}(a)w_{12}(1+a)
x_{12}\left(\frac{a(1+a)}{1-a}\right)x_{12}\left(\frac{a(1+a)}{a-1}\right)\\ 
&=& w_{12}\left(\frac{1+a}{-a}\right)w_{12}(-1)
w_{12}(a)w_{12}(1+a)\\
&=&h_{12}\left(\frac{1+a}{-a}\right)h_{12}(a)h_{12}(-(1+a))^{-1}\\
&=& c(-(1+a)a^{-1},a)=c(1+a,a).
\end{eqnarray*}

Putting this together, we get
\begin{eqnarray*}
\bar{f}(\tilde{\Theta}_a)= \bar{f}([\tilde{H}_a|G_a])\cdot \bar{f}([D(a)|\tilde{H}_a])^{-1}
= c(1+a,a)c(a,1+a)^{-1}\\
= c(a^2,1+a)=c((-a)^2,1+a)
= c(-a,1+a)c(1+a,-a)=1.\\
\end{eqnarray*}
\end{enumerate} 

\end{proof}
\section{Applications: generators for $\hoz{2}{\spl{2}{\Z[1/m]}}$}\label{sec:h2}
%Fix $n\geq 2$. Let $2=p_1<p_2<\cdots < p_n$ denote the first $n$ prime numbers. 
%Let $N= p_1\cdots p_n$. 

The general principle is the following:
\begin{lem}\label{lem:ui} 
Let $m=q_1q_2\cdots q_n$ where $q_1,\ldots,q_n$ are distinct primes. 
Suppose that the positive integers $u_{2},\ldots,u_{n}\in \Z[1/m]^\times$ 
satisfy the conditions
\begin{enumerate}
\item $u_i$ is a primitive root modulo $q_i$ for $i\geq 2$,
\item When $i\not= j\in \{ 2,\ldots,n\}$,  
\[
q_i^{v_{q_j}(u_i)}\cong 1\pmod{q_j} .
\] 
\end{enumerate}
Then
there is a direct sum decomposition
\[
\tilde{K}_2(2,\Z[1/m])\cong \Z\oplus  \Z/(q_2-1)\oplus \cdots \oplus\Z/(q_n-1)
\]
with the property that infinite cyclic factor is generated by  $c(-1,-1)$ and 
the factor $\Z/(q_i-1)$  is generated by  $c(u_{i},q_i)$. 
\end{lem}
\begin{proof}
The isomorphism 
\[
\tilde{K}_2(2,\Z[1/m])\cong \Z\oplus\left(\oplus_{i=2}^n\F{q_i}^\times\right)
\]
is induced by the map 
\[
\sigma: \tilde{K}_2(2,\Z[1/m])\to \Z,\quad 
c(a,b)\mapsto
\left\{
\begin{array}{ll}
1,& a<0\mbox{ and }b<0\\
0,& \mbox{ otherwise}
\end{array}
\right.
\]
and the tame symbols $T_{p_i}:K(2,\Q)\to\F{p_i}^\times$.

%Let $b$ denote the least common multiple of the numbers $p_2-1,\ldots,p_n-1$. 

%For each $i$ choose $w_i<p_i$ a primitive root modulo $p_i$. Observe that 
%$w_i\in \Z[1/N]^\times$ since its prime factors must lie among the $p_j$. 

%Let $y_{i}=(p_1\cdots p_n)/p_i$ and 
% let 
%\[
%u_{i}=w_i\cdot y_{i}^{b}.
%\]

Now
\begin{eqnarray*}
T_{p_i}(c(u_{i},q_i))=\tau_{q_i}(\{ u_i,q_i\})=u_i\pmod{q_i}=w_i
\end{eqnarray*}
while for $j\not=i$
\begin{eqnarray*}
T_{q_j}(c(u_{i},q_i))=q_i^{v_{q_j}(u_i)}{\pmod{q_j}}\equiv 1\pmod{q_j}.
\end{eqnarray*}
\end{proof}
\begin{rem} It is not known whether there must exist units satisfying condition (1) in general, 
but exceptions, if they exist, are rare.

If units $u_i$ are found satisfying condition (1), then it can always be arranged 
for condition (2) to hold; namely, multiply $u_i$ by a high power of $m_i$ where 
$m_i= (\prod_{k=1}^nq_k)/q_i$.  

\end{rem}

Combining Lemma \ref{lem:ui} with Theorems \ref{thm:z} and \ref{thm:symb}, we deduce:

\begin{cor} Let $m=q_1\cdots q_n$ be distinct primes satisfying 
$q_1<q_2<\cdots <q_n$ and $q_1=2,q_2=3$. Let $u_2,\ldots,u_n$ be as in Lemma \ref{lem:ui}.  
There is a direct sum decomposition
\[
\hoz{2}{\spl{2}{\Z[1/m]}}\cong \Z\oplus \left(\oplus_{i=2}^n \Z/(q_i-1)\Z\right) 
\]
where the first summand corresponds to  the subgroup of 
$\hoz{2}{\spl{2}{\Z[1/m]}}$ generated by the homology class $C(-1,-1)$, 
and the summand $\Z/(q_i-1)\Z$ corresponds to the subgroup generated by the homology class 
$C(u_{i},q_i)$.  
\end{cor}
\begin{exa}
In the case $m=6$, we can take $u_2=2$. We deduce that the cyclic factors of 
\[
\hoz{2}{\spl{2}{\Z[1/6]}}\cong \Z\oplus\Z/2
\]
are generated by the homology classes $C(-1,-1)$ and $C(2,3)$.
\end{exa}
\begin{exa} In the case $m=30$, then the units $u_2=2$, $u_3=2$ satisfy the necessary 
congruences. Thus the cyclic factors of 
\[
\hoz{2}{\spl{2}{\Z[1/30]}}\cong \Z\oplus\Z/2\oplus\Z/4
\]
are generated by the homology classes $C(-1,-1)$, $C(2,3)$ and $C(2,5)$. 
\end{exa}

\begin{exa}
By Theorem \ref{thm:z}, we have 
\[
\hoz{2}{\Z[1/42]}\cong \Z\oplus \F{3}^\times\oplus \F{7}^\times\cong\Z\oplus\Z/2\oplus\Z/6.
\] 
The first factor is generated by the homology class $C(-1,-1)$.  Furthermore, $u_2=2=u_3$ satisfy the 
congruences of Lemma \ref{lem:ui}.
It 
follows that the homology classes $C(2,3)$ and $C(2,7)$ generate the second and third cyclic factors. 
\end{exa}
\begin{exa} Let $\omega$ be a primitive cube root of unity and let $p$ be a rational prime which 
is congruent to $1$ modulo $3$. Let $\mathcal{O}=\Z[\omega,\frac{1}{3p}]$. 

Observe that $\omega\in 
\mathcal{O}^\times$ and 
$\omega^2-1= \sqrt{-3}\omega\in \mathcal{O}^\times$ also. Then 
$p\mathcal{O}=\mathfrak{p}_1\mathfrak{p}_2$ where $k(\mathfrak{p}_i)\cong \F{p}$ for $i=1,2$. 
Since $K_2(\Z[\omega])=0$, we have 
\begin{eqnarray*}
K_2(\mathcal{O})\cong \tilde{K}_2(2,\mathcal{O})&\cong &k(\mathfrak{p}_1)^\times\oplus
k(\mathfrak{p}_2)^\times\oplus k(\mathfrak{q})^\times\\
& \cong& \F{p}^\times\oplus\F{p}^\times\oplus\F{3}^\times \\
\end{eqnarray*}
where $\mathfrak{q}=\sqrt{-3}\mathcal{O}$.

By Lemma \ref{lem:h2o} and Theorem \ref{thm:symb} the natural map 
\[
\hoz{2}{\spl{2}{\mathcal{O}}}\to K_2(\mathcal{O})
\]
is surjective and the homology class $C(-\omega,p)$  
maps, via the tame symbol, 
 to the element $-\bar{\omega} \in k(\mathfrak{p}_i)^\times$ of order $6$, while the class 
$C(3,p)$ maps to $\bar{3}\in k(\mathfrak{p}_i)^\times\cong \F{p}^\times$. 
\end{exa}
\bibliography{h2sl2}

\begin{thebibliography}{10}

\bibitem{adem:naffah}
Alejandro Adem and Nadim Naffah.
\newblock On the cohomology of {${\rm SL}\sb 2(\bold Z[1/p])$}.
\newblock In {\em Geometry and cohomology in group theory ({D}urham, 1994)},
  volume 252 of {\em London Math. Soc. Lecture Note Ser.}, pages 1--9.
  Cambridge Univ. Press, Cambridge, 1998.

\bibitem{hut:laurent}
Kevin Hutchinson.
\newblock {Low-dimensional homology of $\mathrm{SL}_2$ of Laurent polynomials}.
\newblock {\em {arXiv:1404.5896}}.

\bibitem{hut:cplx13}
Kevin Hutchinson.
\newblock A {B}loch-{W}igner complex for {$\mathrm{SL}\sb 2$}.
\newblock {\em J. K-Theory}, 12(1):15--68, 2013.

\bibitem{hutchinson:tao3}
Kevin Hutchinson and Liqun Tao.
\newblock Homology stability for the special linear group of a field and
  {M}ilnor-{W}itt ${K}$-theory.
\newblock {\em Doc. Math.}, (Extra Vol.):267--315, 2010.

\bibitem{liehl}
Bernhard Liehl.
\newblock On the group {${\rm SL}\sb{2}$} over orders of arithmetic type.
\newblock {\em J. Reine Angew. Math.}, 323:153--171, 1981.

\bibitem{mat:pres}
Hideya Matsumoto.
\newblock Sur les sous-groupes arithm\'etiques des groupes semi-simples
  d\'eploy\'es.
\newblock {\em Ann. Sci. \'Ecole Norm. Sup. (4)}, 2:1--62, 1969.

\bibitem{mazz:sus}
A.~Mazzoleni.
\newblock A new proof of a theorem of {S}uslin.
\newblock {\em $K$-Theory}, 35(3-4):199--211 (2006), 2005.

\bibitem{mil:intro}
J.~Milnor.
\newblock {\em Introduction to Algebraic {K}-Theory}.
\newblock Annals of Mathematics Studies No. 72. Princeton University Press,
  1971.

\bibitem{milnor:quad}
John Milnor.
\newblock Algebraic {$K$}-theory and quadratic forms.
\newblock {\em Invent. Math.}, 9:318--344, 1969/1970.

\bibitem{moore:pres}
Calvin~C. Moore.
\newblock Group extensions of {$p$}-adic and adelic linear groups.
\newblock {\em Inst. Hautes \'Etudes Sci. Publ. Math.}, (35):157--222, 1968.

\bibitem{morel:trieste}
Fabien Morel.
\newblock An introduction to {$\Bbb A\sp 1$}-homotopy theory.
\newblock In {\em Contemporary developments in algebraic $K$-theory}, ICTP
  Lect. Notes, XV, pages 357--441 (electronic). Abdus Salam Int. Cent. Theoret.
  Phys., Trieste, 2004.

\bibitem{morel:puiss}
Fabien Morel.
\newblock Sur les puissances de l'id\'eal fondamental de l'anneau de {W}itt.
\newblock {\em Comment. Math. Helv.}, 79(4):689--703, 2004.

\bibitem{morita:zs}
Jun Morita.
\newblock On the group structure of rank one {$K\sb 2$} of some {${\bf Z}\sb
  S$}.
\newblock {\em Bull. Soc. Math. Belg. S\'er. A}, 42(3):561--575, 1990.
\newblock Algebra, groups and geometry.

\bibitem{serre:sl2}
Jean-Pierre Serre.
\newblock Le probl\`eme des groupes de congruence pour {SL}2.
\newblock {\em Ann. of Math. (2)}, 92:489--527, 1970.

\bibitem{serre:trees}
Jean-Pierre Serre.
\newblock {\em Trees}.
\newblock Springer-Verlag, Berlin, 1980.
\newblock Translated from the French by John Stillwell.

\bibitem{steinberg:chev}
Robert Steinberg.
\newblock {\em Lectures on {C}hevalley groups}.
\newblock Yale University, New Haven, Conn., 1968.
\newblock Notes prepared by John Faulkner and Robert Wilson.

\bibitem{sus:tors}
A.~A. Suslin.
\newblock Torsion in {$K\sb 2$} of fields.
\newblock {\em $K$-Theory}, 1(1):5--29, 1987.

\bibitem{swan:special}
Richard~G. Swan.
\newblock Generators and relations for certain special linear groups.
\newblock {\em Bull. Amer. Math. Soc.}, 74:576--581, 1968.

\bibitem{tuanellis}
Bui~Anh Tuan and Graham Ellis.
\newblock The homology of {$SL\sb 2(\Bbb Z[1/m])$} for small {$m$}.
\newblock {\em J. Algebra}, 408:102--108, 2014.

\bibitem{vanderkallen:stab}
Wilberd van~der Kallen.
\newblock Stability for {$K\sb{2}$} of {D}edekind rings of arithmetic type.
\newblock In {\em Algebraic {$K$}-theory, {E}vanston 1980 ({P}roc. {C}onf.,
  {N}orthwestern {U}niv., {E}vanston, {I}ll., 1980)}, volume 854 of {\em
  Lecture Notes in Math.}, pages 217--248. Springer, Berlin, 1981.

\bibitem{vaserstein:sl2}
L.~N. Vaser{\v{s}}te{\u\i}n.
\newblock The group {$SL\sb{2}$} over {D}edekind rings of arithmetic type.
\newblock {\em Mat. Sb. (N.S.)}, 89(131):313--322, 351, 1972.

\end{thebibliography}
\end{document}